\documentclass{amsart}
\usepackage[toc,page]{appendix}
\usepackage{amssymb,epsfig,amsxtra, amsmath}
\usepackage{amscd} 
\usepackage{xypic}
\usepackage[mathscr]{eucal}
\usepackage{a4}
\input{xy}
\xyoption{all}
\newtheorem{theorem}{\textbf{Theorem}}[section]
\newtheorem{proposition}[theorem]{\textbf{Proposition}}
\newtheorem{definition}[theorem]{\textbf{Definition}}
\newtheorem{lemma}[theorem]{\textbf{Lemma}}
\newtheorem{corollary}[theorem]{\textbf{Corollary}}
\newtheorem{remark}[theorem]{\textbf{Remark}}
\newtheorem{example}[theorem]{\textbf{Example}}
\newtheorem{definition-remark}{\textbf{Definition-Remark}}

\newtheorem{claim}[theorem]{\textbf{Claim}}

\def\ZZ{{\mathbb Z}}

\def\PP{{\mathbb P}}

\def\O{{\mathcal O}}
\def\N{{\mathcal N}}

\def\A{{\mathcal A}}

\def\U{{\mathcal U}}
\def\K{{\mathcal K}}
\def\Pic{\operatorname{Pic}}
\def\Sing{\operatorname{Sing}}
\def\deg{\operatorname{deg}}

\def\d{\delta}

\def\ag{\`a}

\def\p{\rm p}

\begin{document}

\title[Curves with $A_k$-singularities on $K3$ surfaces]{On the existence of curves with 
$A_k$-singularities on $K3$ surfaces}
\author{Concettina Galati }
\address{Dipartimento di Matematica\\
 Universit\ag\, della Calabria\\
via P. Bucci, cubo 31B\\
87036 Arcavacata di Rende (CS), Italy. }
\email{galati@mat.unical.it }
\author{Andreas Leopold Knutsen}
\address{Department of Mathematics, University of Bergen, Postboks 7800, 5020  
Bergen, Norway.}
\email{andreas.knutsen@math.uib.no}

\thanks{Both authors want to thank
the Department of Mathematics of the University of Calabria and the Department of Mathematics of the University
of Bergen for hospitality and for financial support. The first author was also supported by GNSAGA of INdAM
and by the PRIN 2008 'Geometria delle variet\ag\,  algebriche e dei loro spazi di moduli', co-financied by MIUR}

\subjclass{14B07, 14H10, 14J28}

\keywords{versal deformations, tacnodes, Severi varieties, $K3$ surfaces, $A_k$-singularities}

\date{August 1st, 2014}

\dedicatory{}

\commby{}


\begin{abstract}
Let $(S,H)$ be a general primitively polarized $K3$ surface. We prove the existence
of irreducible curves in $|\mathcal O_S(nH)|$ with $A_k$-singularities and corresponding to regular points of the 
equisingular deformation locus. Our result is optimal for $n=1$. As a corollary, we get the existence of 
irreducible curves in $|\mathcal O_S(nH)|$ of geometric genus $g\geq 1$ with a cusp and nodes
or a simple tacnode and nodes.  We obtain our result by studying the versal deformation
family of the $m$-tacnode.  Moreover, using results of Brill-Noether theory on curves of $K3$ surfaces,
we provide a regularity condition
for families of curves with only $A_k$-singularities in $|\mathcal O_S(nH)|.$   
\end{abstract}


\maketitle

\section{Introduction}\label{introduction}
Let $S$ be a complex smooth projective $K3$ surface and let $H$ be a globally generated line bundle 
of sectional genus ${\p}={\p}_a(H) \geq 2$  and such that $H$ is not divisible in $\Pic S$. The pair
$(S,H)$ is called a {\it primitively polarized $K3$ surface of genus ${\p}.$} 
It is well-known that the moduli space 
$\K_{{\p}}$ of primitively polarized $K3$ surfaces of genus ${\p}$
is non-empty, smooth and irreducible of dimension $19.$ 
Moreover, if $(S,H)\in \K_{{\p}}$ is a very general element
(meaning that it belongs to the complement of a countable union of Zariski closed proper subsets), then $\Pic S \cong \ZZ[H].$
If $(S,H)\in \K_{{\p}}$, we denote
by $\mathcal V_{nH,1^\delta}^S\subset |\mathcal O_S(nH)|=|nH|$ the so called Severi variety
of $\delta$-nodal curves, defined as the Zariski closure of the locus of
{\it irreducible} and reduced curves with exactly 
$\d$ nodes as singularities. More generally, we will denote by 
$\mathcal V_{nH,\small1^{d_2},\,2^{d_3},\ldots,\,{(m-1)}^{d_m}}^{S}$  the Zariski closure 
of the locus in $|nH|$ of reduced and irreducible curves with exactly $d_k$ singularities
of type $A_{k-1},$ for every $2\leq k\leq m$,  and no further singularities.
We recall that an $A_k$-singularity is a plane singularity of analytic equation $y^2-x^{k+1}.$
Every plane singularity of multiplicity $2$ is an $A_k$-singularity, for some 
$k.$ 

The Severi variety $\mathcal V_{nH,1^\delta}^S\subset |\mathcal O_S(nH)|$ is a 
well-behaved variety.  By \cite{tan}, we know that $\mathcal V_{nH,1^\delta}^S$
is smooth of the expected dimension at every point $[C]\in\mathcal V_{nH,1^\delta}^S$ corresponding to a $\d$-nodal curve,
i.e. the tangent space $T_{[C]}\mathcal V_{nH,1^\delta}^S$ has dimension $\dim(|nH|)-\delta$ 
for every $\d\leq\dim(|nH|)={\p}_a(nH)$. 
The existence of nodal curves of every allowed genus in the primitive linear system $|H|$ on a general primitively polarized K3 
surface has been proved first by Mumford, cf. \cite{MM}. Later Chen proved the non-emptiness of $\mathcal V_{nH,1^\delta}^S$ 
in the case $(S,H)$ is a general primitively polarized $K3$ surface, 
$n\geq 1$ and $\d\leq\dim(|nH|)={\p}_a(nH)$ \cite{C}. Chen's existence theorem is obtained by degeneration 
techniques. A very general primitively polarized $K3$ surface $S_t\subset\PP^{{\p}}$ of genus ${{\p}}$ 
is degenerated in $\PP^{{\p}}$ to the union of two rational normal scrolls $S_0=R_1\cup R_2,$ 
intersecting transversally along a smooth elliptic normal curve $E.$
Rational nodal curves on $S_t$ are obtained by deformation from suitable reduced curves 
$C_0=C^1\cup C^2\subset S_0$ having tacnodes at points of $E$ and nodes elsewhere. A key 
ingredient in the proof of Chen's theorem is the Caporaso-Harris description of the locus of $(m-1)$-nodal 
curves in the versal deformation space $\Delta_m$ of the 
$m$-tacnode (or $A_{2m-1}$-singularity). The question we ask in this paper is the following.

\begin{trivlist}
\item[\hspace\labelsep{\bf Main Problem.}]\sl
{\it With the notation above, assume that $C=C_1\cup C_2\subset R_1\cup R_2$
is any curve having an $m$-tacnode at a point $p$ of $E.$ Then, which kinds of curve singularities
on $S_t$ may be obtained by deforming the $m$-tacnode of $C$ at $p?$} 
\end{trivlist}

Theorem \ref{prop:defcusps}, which is to be considered the main result of this article, completely 
answers this question. It proves that, under suitable hypotheses,
 {\it the $m$-tacnode of $ C$ at $p$ deforms to $d_k$ singularities of type $A_{k-1},$ for every
$2\leq k\leq m$ and $d_k\geq 0$ such that $\sum_kd_k(k-1)=m-1.$} By trivial dimensional reasons, no further 
singularities on $S_t$ may be obtained by deforming the $m$-tacnode of $C\subset R_1\cup R_2.$
The result is a local result, obtained by studying the versal deformation
family of the $m$-tacnode, with the same approach as in \cite[Section 2.4]{ch}. In particular, the result holds for
any flat family $\mathcal X\to\Delta$ of regular surfaces, with smooth total space $\mathcal X$ and special fibre $\mathcal X_0=A\cup B$
having two irreducible components $A$ and $B$ intersecting transversally, and it can be applied to curves $C^\prime\subset\mathcal X_0$ with
several tacnodes on $E$ and any kind of singularities on $\mathcal X_0\setminus E$, cf. Corollary \ref{ingenerale} and Remark \ref{global-use}.
Section \ref{Smoothing tacnodes} is completely devoted to the proof of Theorem \ref{prop:defcusps}.
In Section 4, inspired by \cite{C}, we apply Theorem \ref{prop:defcusps}, more precisely Corollary \ref{ingenerale},  
to a family of $K3$ surfaces with  suitable central fibre fibre $\mathcal X_0=R_1\cup \tilde{R}_2$, 
by deforming curves $C_0=C^1\cup C^2\subset R_1\cup \tilde{R}_2$ ad hoc constructed, and we 
obtain the following result. 

\begin{theorem} \label{thm:maink3}
Let $(S,H)$ be a general primitively polarized $K3$ surface of genus ${\p}={\p}_a(H)=2l+\epsilon \geq 3,$
with $l\geq 1$ and $\epsilon=0,1.$
 Then, for every $n\geq 1$ and for every $(m-1)$-tuple of non-negative integers $d_2,\ldots,d_m$ satisfying
 \begin{equation}\label{m-1}
\sum_{k=2}^m(k-1)d_k=\left\{\begin{array}{l}
2n(l-1+\epsilon)+2-\epsilon, \,\mbox{if}\,\, (n,{\p})\neq (2,3),\,(2,4),\\
2n(l-1+\epsilon)+1-\epsilon,\,\mbox{if}\,\, (n,{\p})= (2,3),\,(2,4),\\
\end{array}\right.
 \end{equation}
there exist reduced irreducible curves $C$ in the linear system $|nH|$ on $S$ 
 such that:
 \begin{itemize}
 \item $C$ has $d_k$ singularities of type $A_{k-1},$ for every $k=3,\ldots,m,$
 and $\delta+d_2$ nodes, where $\delta=\dim(|nH|)-\sum_{k=2}^m(k-1)d_k,$  and no further singularities;
 \item  $C$ corresponds to a regular point of the equisingular deformation locus $ES(C).$ Equivalently, 
 $\dim(T_{[C]}ES(C))=0.$ 
 \end{itemize} 
 Finally, the singularities of $C$ may be smoothed independently. In particular, under the hypothesis \eqref{m-1}, 
 for any $d_k'\leq d_k$ and for any $\delta ' \leq \delta$,
 there exist curves $C$ in the linear system $|nH|$ on $S$ with $d'_k$ singularities of type $A_{k-1},$
for every $k=3,\ldots,m,$ and $\delta'+d_2'$ nodes as further singularities and corresponding to regular points of their equisingular 
 deformation locus.
\end{theorem}

The notion of equisingular deformation locus and regularity is recalled in Definition \ref{definition-regularity} and Remark \ref{A_k}.
In Corollaries  \ref{H}  and \ref{divisori} we observe that Theorem \ref{thm:maink3} is optimal if $n=1$
and that, for $n\geq1,$ it implies the existence of curves of every geometric genus $g\geq 1$ with
a cusp and nodes or a $2$-tacnode and nodes as further singularities. By \cite{C1}, 
this is not possible for $(g,n)=(0,1)$ . Finally, in the next section, we recall some standard 
results and terminology of deformation theory that will be useful later, 
focusing our attention on properties of equisingular deformations of curves with only $A_k$-singularities
on $K3$ surfaces. In Section \ref{tangentspace} we also provide the following regularity condition. 

\begin{proposition}  \label{prop:regolarita-gen}
Let $S$ be a $K3$ surface with $\Pic S \cong \ZZ[H]$, let ${\p}={\p}_a(H)$ and $n \geq 1$ an integer. 
Assume that $C \in |nH|$ is a reduced and irreducible curve on $S$ having 
precisely $d_k\geq 0$ singularities of type $A_{k-1}$, for each $k \geq 2$, and no further singularities, such that 
\begin{eqnarray}
\sum_{k}(k-1)d_k&=&\deg T^1_C< \frac{\p+2}{2}=\frac{H^2}{4}+2,\;\;\;\;\mbox{if}\;\;n=1\;\;\mbox{ or}\label{n=1}\\ 
\sum_{k}(k-1)d_k&=&\deg T^1_C  < 2(n-1)(\p-1)=(n-1)H^2,\;\;\mbox{if}\;\;   n \geq 2, \label{nalmeno2}
\end{eqnarray}
where $T^1_C$ is the first cotangent bundle of $C.$
Then $[C]$ is a regular point of $ES(C)$ and the singularities of $C$ may be smoothed independently.
\end{proposition}

The previous proposition is obtained by results of Brill-Noether theory on curves of $K3$ surfaces \cite{laz, kn1, gom1}. 
In particular, its proof does not require any degeneration argument 
of surfaces or curve singularities and is thus independent of the other results in this paper. 
Proposition \ref{prop:regolarita-gen} together with Theorem \ref{thm:maink3}  provide sufficient conditions
for the variety $\mathcal V_{nH,\small1^{d_2},\,2^{d_3},\ldots,\,{(m-1)}^{d_m}}^{S}$ to be non-empty and regular,
see Remark \ref{ex-reg-1}.

\subsection*{Acknowledgments} 
The first author is indebted with J. Harris for invaluable conversations
on deformation theory of curve singularities. She also benefited from conversations 
with F. van der Wyck. Both authors want to 
express deep gratitude to C. Ciliberto and T. Dedieu for many stimulating questions
and suggestions. 
Finally the authors are very grateful to the referee for his careful reading and many comments. He pointed out 
a mistake in the original version of the paper and provided a series of suggestions that ultimately  
improved the exposition and readability.  

\section{Tangent spaces and a new regularity condition}\label{tangentspace}
In this section we recall some properties of the equisingular and equigeneric
deformation loci of a reduced curve on an arbitrary smooth projective $K3$ surface $S$
and, in particular, of a curve with only $A_k$-singularities. Finally, at the end of the section, 
we prove Proposition \ref{prop:regolarita-gen}.

Let $S$ be a smooth projective $K3$ surface and let $D$ be a Cartier divisor on $S$ of arithmetic genus
${\p}_a(D).$ Assume that $|D|=|\mathcal O_S(D)|$ is a Bertini linear system,
i.e. a linear system without base points and whose general element corresponds to a smooth
curve. (In fact, by \cite{saint-donat}, every irreducible curve $D$ on $S$ such that $D^2\geq 0$ defines a
Bertini linear system on $S.$) If $C\in |D|$ is a reduced curve, we consider 
the following standard exact sequence of sheaves on $C$
 \begin{equation}\label{sequenceoneone}
\xymatrix{
0  \ar[r] &      \Theta_C  \ar[r] &      \Theta_{S}|_C \ar[r]^{\alpha} & \mathcal N_{C|S} \ar[r]^{\beta} &   T^1_C  \ar[r] &     0, }
 \end{equation}
 where $\Theta_C = \mathfrak{hom}(\Omega_C^1,\mathcal O_C)$ is the tangent sheaf of $C,$
  $\Theta_{S}|_C$ is the tangent sheaf of $S$ restricted to $C$, $\mathcal N_{C|S}\simeq \mathcal O_C(C)$ is the normal
   bundle of $C$ in $S$, and  $T^1_C$ is the first cotangent sheaf of $C$.  The latter is supported on the singular 
  locus $\Sing(C)$ of $C$, and its stalk $T^1_{C,p}$ at every singular point
 $p$ of $C$  is the versal deformation space of the singularity (see \cite[(3.1)]{diaz_harris}, \cite{ser}, \cite{hm} or \cite{GLS}).
Identifying $H^0(C,\N_{C|S})$ with the tangent space $T_{[C]}|D|$, 
the induced map 
\begin{equation}\label{dalpha}
\xymatrix{
H^0(\beta):\; H^0(C,\N_{C|S})\ar[r] &  H^0(C,T^1_C)=\oplus_{p\in \Sing(C)}T^1_{C,p}}
\end{equation}
is classically identified with the differential at $[C]$ of the versal map from an analytic neighborhood
of $[C]$ in $|D|$ to an analytic neighborhood of the origin in $H^0(C,T^1_C).$ 
By this identification and by the fact that the origin in $T^1_{C,p}$ is the only point parametrizing singularities 
analytically equivalent to the singularity of $C$ at $p$ \cite[Lemma (3.21)]{diaz_harris}, we have that the 
global sections of the kernel  $\N^\prime_{C|S}$ of the sheaf map $\beta$ in \eqref{sequenceoneone} 
are infinitesimal deformations of $C$ that are analytically equisingular, i.e. infinitesimal deformations of $C$ preserving the analytic class of every
singularity of $C$ \cite[Definition (3.9)]{diaz_harris}.  For this reason, $\N^\prime_{C|S}$ is usually called the \textit{equisingular normal 
sheaf of $C$ in $S$} \cite[Prop. 1.1.9 (ii)]{ser}.  Let $J$ be the Jacobian ideal of $C$. 
By a straightforward computation,
$J\otimes\N_{C|S}=\N^\prime_{C|S}$ and, consequently, $\dim(H^0(C,T^1_C))=\deg(J)=\sum_{p\in C}\deg (J_p),$
where $J_p$ is the localization of $J$ at $p$. Keeping in mind the versal property of $T^1_C,$
the following definition makes sense.  

\begin{definition}\label{jacobian}
We say that the singularities of $C$ may be smoothed independently if the map
$H^0(\beta)$ in \eqref{dalpha}
is surjective or, equivalently, if $h^0(C, \N^\prime_{C|S})=h^0(C,\N_{C|S})-\deg(J).$
If this happens, we also say that the Jacobian ideal imposes linearly independent conditions
to the linear system $|D|.$
\end{definition}

\begin{remark}\label{smoothing_independently}
If $C$ is an irreducible reduced curve in a Bertini linear system $|D|$ on a smooth projective
$K3$ surface $S$, then $h^1(C,\N_{C|S})=h^1(C,\mathcal O_C(C))=h^1(C,\omega_C)=1,$ where $\omega_C$
denotes the dualizing sheaf of $C.$ In particular, by the short exact sequence of sheaves on $C$
$$
\xymatrix{0\ar[r]&\N_{C|S}^\prime\ar[r]&\N_{C|S}\ar[r]& T^1_C\ar[r]&0,}
$$
we have that $h^1(C,\N^\prime_{C|S})\geq 1$, and the singularities of $C$ may be smoothed independently
if and only if $h^1(C,\N^\prime_{C|S})= 1.$
\end{remark}
The locus in $|D|$ of deformations of $C$ preserving the analytic class of singularities coincides with the locus
of formally locally trivial deformations in the Zariski topology or locally trivial deformations in the \'etale topology
\cite[Proposition (3.23)]{diaz_harris}.  In general this locus is a proper subset of the Zariski locally closed subset  
$ES(C)\subset |D|$ parametrizing topologically equisingular deformations of $C$ in $|D|,$ i.e. deformations of $C$
preserving the equisingular class of every singularity of $C.$  For the notion of
equisingular deformation of a plane singularity, we refer to \cite[Definition (3.13)]{diaz_harris}.
The {\it equisingular deformation locus} $ES(C)$ of $C$ in $|D|$ has a natural structure of scheme, representing
a suitable deformation functor \cite[Section 2]{GL}. The tangent space $T_{[C]}ES(C)$ to $ES(C)$ 
at the point $[C]$ corresponding to $C,$ is well understood. In particular, there exists an ideal sheaf $I,$
named the {\it equisingular ideal} of $C,$ such that $J\subset I$ and
$$T_{[C]}ES(C)\simeq H^0(C, I\otimes \mathcal O_C(C)).$$ 

\begin{definition}\label{definition-regularity}
We say that $[C]$ is a regular point of $ES(C)$ if $ES(C)$ is smooth of the expected dimension
at $[C],$ equivalently if 
$$\dim(T_{[C]}ES(C))=\dim(H^0(C, I\otimes \mathcal O_C(C)))=\dim(H^0(C, \mathcal O_C(C)))
-\deg I.$$
In this case, we also say that the equisingular ideal imposes linearly independent conditions to 
curves in $|D|.$ 
\end{definition}
We also recall the inclusion $J\subset I\subset A,$  where $A$ is the conductor ideal.

Throughout this paper we will be interested in curves with $A_k$-singularities. An $A_k$-singularity has 
analytic equation  $y^2=x^{k+1}.$ Every plane curve singularity of multiplicity $2$ is an $A_k$-singularity
for a certain $k\geq 1.$ In particular, two 
singularities of multiplicity $2$ are analytically equivalent if and only if they are topologically equivalent.   

\begin{remark}\label{A_k}
The equisingular ideal $I$ of an $A_k$-singularity of equation $y^2=x^{k+1}$ 
coincides with the Jacobian ideal $J=I=(y,x^{k})$ \cite[Proposition 6.6]{wahl}.
It follows that, if $C\in |D|$ is a reduced curve on $S$ with only $A_k$-singularities, $ W\subset |D|$ is the linear system 
of curves passing  through every $A_k$-singularity $p\in C$ and tangent there to the reduced 
tangent cone to $C$ at $p$ with multiplicity $k$ and $\mathcal W\subset H^0(S,\O_S(D))$ is the vector space such that
$\mathbb P(\mathcal W)=W$, then the tangent space 
$$T_{[C]}ES(C)\simeq H^0(C,\N_{C|S}\otimes I)=H^0(C,\N_{C|S}\otimes J)
=H^0(C,\N_{C|S}^\prime)$$ to $ES(C)$ at the point $[C]$ is isomorphic to $r_C(\mathcal W)$, where
$r_C: H^0(S,\O_S(D))\to H^0(C,\O_C(D))$ is the natural restriction map. In particular, every $A_k$-singularity imposes at most 
$k=\dim(\mathbb C[x,y]/(y,x^{k}))$ linearly independent conditions to $|D|$, and the equisingular 
deformation locus $ES(C)$ of $C$ in $|D|$ 
is regular at $[C]$ if and only if the singularities of $C$ may be smoothed independently.
If $C\in |D|$ is reduced and irreducible with $d_k$ singularities of type $A_{k-1}$, $k=2,...,m$, and no further singularities, then  
the reduced support of $ES(C)$ is an open set in one irreducible component $V$ of the variety 
$\mathcal V_{D,\small1^{d_2},\,2^{d_3},\ldots,\,{(m-1)}^{d_m}}^{S}$ introduced in Section \ref{introduction}. 
In particular we have that 
$$
T_{[C]}V\subset T_{[C]}ES(C)\simeq H^0(C,\N_{C|S}^\prime).
$$
We say that $V$ is regular at $[C]$ if $ES(C)$ is regular at $[C]$,  in which case we have that
$T_{[C]}V=T_{[C]}ES(C)$ and $\dim(T_{[C]}V)=\dim(T_{[C]}ES(C))=h^0(C,\N_{C|S}^\prime)=\dim(|D|)-\sum_kd_k(k-1).$
Moreover $V$ is said to be regular if it is regular at every point corresponding to
an irreducible and reduced curve with $d_k$ singularities of type $A_{k-1}$, $k=2,\dots,m$, and no further singularities.
Finally, we say that  $\mathcal V_{D,\small1^{d_2},\,2^{d_3},\ldots,\,{(m-1)}^{d_m}}^{S}$ is regular if all its irreducible
components are regular. In particular, if $\mathcal V_{D,\small1^{d_2},\,2^{d_3},\ldots,\,{(m-1)}^{d_m}}^{S}$ is regular,
all its irreducible components are generically smooth of the expected dimension.
\end{remark}  

If $k$ is odd, an $A_k$-singularity is also called a $\frac{k+1}{2}$-tacnode whereas, if $k$ is even, an 
$A_k$-singularity is said to be a cusp. Moreover, by classical terminology, $A_1$-singularities are nodes, 
$A_2$-singularities are ordinary cusps and $A_3$-singularities are called simple tacnodes.
As we already observed, for every $\delta\leq \p_a(D)$, the Severi variety $\mathcal V_{D,1^\delta}^{S}$ of $\delta$-nodal curves is a regular variety, i.e. 
is smooth of the expected dimension at 
every point $[C]$ corresponding to a curve with exactly $\delta$ nodes as singularities \cite{tan}.

Now we may prove our regularity condition
for curves with only $A_k$-singularities on a $K3$ surface $S$ with $\Pic S \cong \ZZ[H]$.

\begin{proof}[Proof of Proposition \ref{prop:regolarita-gen}.]
Assume that $[C]$ is {\it not} a regular point of $ES(C).$ Then, by Remarks \ref{smoothing_independently} and
\ref{A_k}, we must have $h^1(\N'_{C|S}) \geq 2.$ Now consider $\N'_{C|S}$ as a torsion sheaf on $S$ and  
define $\A:= \mathfrak{ext}^1(\N'_{C|S}, \O_S)$. Then $\A$ is a rank one torsion free sheaf on $C$
 and a torsion sheaf on $S.$ Moreover, by \cite[Lemma 2.3]{gom1}, being $S$ a $K3$ surface, we have that 
 $h^0(\A)=h^1(\N'_{C|S}) \geq 2$ and 
\[\deg \A = C^2-\deg \N'_{C|S}= \deg T^1_C=\sum_{k}(k-1)d_k.\]
 By \cite[Prop.~ 2.5 and proof of Thm.~I at p.~749]{gom1}, the pair $(C,\A)$ may be deformed to a pair $(C',\A')$ where 
 $C' \sim C$ is smooth, and $\A'$ is a line bundle on $C'$ with $h^0(\A') \geq h^0(\A)$ and $\deg \A'=\deg \A$. In other 
 words, there is a smooth curve in $|nH|$ carrying a $g^1_{\deg T^1_C}$. If $n=1$ then, by Lazarsfeld's famous 
 result \cite[Cor.~1.4]{laz}, no curve in $|H|$ carries any $g^1_d$ with $2d < {\p}_a(H)+2.$ 

Now assume that $n \geq 2$. By \cite[Theorem 1.3]{kn1}, the minimal gonality
 of a smooth curve in a complete linear system $|L|$ on any $K3$
 surface is either
$\lfloor \frac{p_a(L)+3}{2} \rfloor =\lfloor \frac{L^2}{4} \rfloor +2$
(the gonality of a generic curve of genus $p_a(L)$) or
the minimal integer $d$ such that $2 \leq d < \lfloor
\frac{p_a(L)+3}{2} \rfloor$ and there is an effective divisor $D$
satisfying $D^2 \geq 0$,
$(L^2,D^2)  \neq (4d-2, d-1)$ and
\[ 2D^2 \stackrel{(i)}{\leq} L.D \leq D^2+d  \stackrel{(ii)}{\leq} 2d, \]
with equality in (i) if and only if $L \sim 2D$ and $L^2 \leq 4d$ and
equality in (ii) if and only if $L \sim 2D$ and $L^2 = 4d$. If $L \sim
nH$ with $n \geq 2$ and $\Pic S \cong \mathbb{Z}[H]$, one easily
verifies that the minimal integer satisfying these conditions is
$d=(n-1)H^2 = 2(n-1)(p_a(H)-1)$ (with $D=H$). The
result follows.
\end{proof}

\begin{remark}
As far as we know, the previously known regularity condition
for curves $C$ as in the statement of the proposition above is given by
\begin{equation}\label{regularity_keilen}
\sum_{k}k^2d_k\leq n^2H^2,
\end{equation}
which has been deduced from \cite[Cor.~2.4]{keilen1}. 
This result is very different from Proposition \ref{prop:regolarita-gen} and we will not
compare the two results here. 
\end{remark}

We conclude this section with a naive upper-bound on the dimension of the equisingular deformation locus
 of an irreducible curve with only $A_k$-singularities on a smooth $K3$ surface. This result is a simple 
 application of Clifford's theorem, and for nodal curves it reduces to Tannenbaum's proof that Severi varieties of irreducible
 nodal curves on $K3$ surfaces have the expected dimension \cite{tan}. 

 \begin{lemma}  \label{prop:dimdasopra-gen}
 Let $|D|$ be a Bertini linear system on a smooth projective $K3$ surface $S.$ 
Let $C\in |D|$ be a reduced and irreducible genus $g$ curve with only $A_k$-singularities,
$\tau$ of which are (not necessarily ordinary) cusps. 
Then 
 \[ \dim T_{[C]}ES(C) \leq g-\tau/2.\]
\end{lemma}

\begin{proof}
Let $C$ and $S$ be as in the statement. By Remark \ref{A_k}, since $C$ has only $A_k$-singularities, we have that  
$T_{[C]}ES(C)=H^0(C,\N^\prime_{C|S}).$  Moreover, by standard deformation theory (see e.g. \cite[(3.51)]{ser}), if 
$\phi:\widetilde{C} \to C\subset S$ is the normalization map,
we have the following exact sequence of line bundles on $\widetilde C$
\begin{equation}\label{senzatorsione}
\xymatrix{0  \ar[r] &  \Theta_{\widetilde C}(Z)     \ar[r]^{\phi_*} &  \phi^*\Theta_{S}    \ar[r] &  { \mathcal N}'_\phi  \ar[r] &  0}.
\end{equation} 
Here $\phi_*:\Theta_{\widetilde C}\to \phi^*\Theta_{S}$ is the differential map of $\phi,$ having zero divisor
$Z$, and  $\mathcal N_\phi^\prime\simeq { \mathcal N}_\phi/\mathcal K_\phi$
is the quotient of the normal sheaf ${\mathcal N}_\phi$ of $\phi$ by its torsion subsheaf $\mathcal K_\phi$
(with support on $Z$).
By \eqref{senzatorsione}, using that $S$ is a $K3$ surface,  we have that
$h^1(\widetilde C,\mathcal N_\phi^\prime)=h^1(\widetilde C,\Theta_{\widetilde C}^{-1}(-Z)  )\geq 1.$
Moreover, again by \cite[p.~174]{ser}, one has
$$\mathcal N^\prime_{\phi} \simeq\phi^*\mathcal N^\prime_{C|S}\,\,\,\textrm{
and\,\,hence}\,\,\, h^0(C,\mathcal N^\prime_{C|S})\leq h^0(\widetilde C,\mathcal N^\prime_{\phi}).$$
Finally, by applying Clifford's theorem, we deduce the desired inequality 
\[
h^0(C,\N'_{C|S}) \leq h^0(\widetilde C,{\N}^\prime_{\phi}) \leq \frac{1}{2}\deg{\N}^\prime_{\phi} +1 = 
\frac{1}{2}(2g-2-\tau)+1 = g-\frac{\tau}{2}. \]
\end{proof}

\section{Smoothing tacnodes}\label{Smoothing tacnodes}
In this section, by using classical deformation theory of plane curve singularities, 
we will find sufficient conditions for the existence of curves with $A_k$-singularities
on smooth projective complex surfaces that we may obtain as deformations of a "suitable" 
reducible surface.

Let $\mathcal X\to\mathbb A^1$ be a flat family of projective surfaces with smooth total space $\mathcal X$. 
Assume moreover that $\mathcal X\to\mathbb A^1$  has smooth and regular general fibre
$\mathcal X_t$ and reducible central fiber $\mathcal X_0=A\cup B$,
consisting of two irreducible components $A$ and $B$ with $h^1(\mathcal O_A)=
h^1(\mathcal O_B)=h^1(\mathcal O_{\mathcal X_t})=0$
and intersecting transversally along a smooth
curve $E=A\cap B.$ Let $D$ be a Cartier divisor on $\mathcal X$. We denote by 
$ D_t=D\cap\mathcal X_t$ the restriction  of $D$ to the fiber $\mathcal X_t.$
Notice that, since $\mathcal X_0=A\cup B$ is a reducible surface, the Picard group 
$\textrm{Pic}(\mathcal X_0)$ of $\mathcal X_0$ is the fiber product of the Picard groups 
$\textrm{Pic}(A)$ and $\textrm{Pic}(B)$ over $\textrm{Pic}(E)$. In particular, we have that
$$
|\mathcal O_{\mathcal X_0}(D)|
=\mathbb P(H^0(\mathcal O_{A}(D))\times_{H^0(\mathcal O_{E}(D))}H^0(\mathcal O_{B}(D))).
$$
From now on, for every  curve $C\subset \mathcal X_0$, we will denote by $C_A$ and $C_B$
the restrictions of $C$ to $A$ and $B$, respectively.
Let $p$ be a point of $E.$ Choose local analytic coordinates $(x,z)$ of $A$ at $p$ 
and $(y,z)$ of $B$ at $p$ in such a way that the equation of $\mathcal X$ at $p$, by using coordinates
$(x,y,z,t)$, is given by $xy=t.$

Now assume there exists a divisor $C=C_A\cup C_B\subset\mathcal X_0$, with $[C]\in|D_0|$, 
such that $C_A$ and $C_B$ are both smooth curves, tangent to $E$ at a point 
$p\in E$ with multiplicity $m$ and intersecting $E$ transversally outside $p$. Local analytic equations 
of $C$ at $p$ are given by 
\begin{eqnarray}\label{tacnode}
\left\{\begin{array}{l}
y+x-z^m=0\\
xy=t \\
t=0,\\
\end{array}\right.
\end{eqnarray}
with $m\geq 2.$ 
Since the singularity of $C$ at $p$ is analytically equivalent to the tacnode
of local equation
\begin{equation}\label{m-tacnode}
f(y,z)=(y-z^{m})y=0,
\end{equation}
we say that \textit{$C$ has an $m$-tacnode at $p$.} 

\begin{definition} \label{W_{p,m}}
We say that the $m$-tacnode of $C$
at $p$ imposes linearly independent conditions to $|D_0|$ if the 
linear system  $W_{p,m}\subset |D_0|$ parametrizing curves $F_A\cup F_B\subset\mathcal X_0,$
such that $F_A$ and $F_B$ are tangent to $E$ at $p$ with multiplicity $m$, has codimension $m$
(which is the expected codimension).   
\end{definition}

\begin{remark}\label{transversality}
We remark that, if the $m$-tacnode of $C$ at $p$ imposes linearly independent conditions to $|D_0|,$
then, for every $r\leq m,$ the locus $W_{p,r}\subset |D_0|$ parametrizing curves with an $r$-tacnode
at $p$ is non-empty of codimension exactly $r.$ In particular, the general element of an analytic neighborhood
of $[C]$ in $|D_0|$ intersects $E$ transversally at $m$ points close to $p.$  
\end{remark}

We now introduce the main result of this paper.

\begin{theorem} \label{prop:defcusps}
Let  $\{d_2,\ldots,d_{m}\}$ be an $(m-1)$-tuple of non-negative integers such that
$$\sum_{j=2}^m(j-1)d_j=m-1.$$ Using the notation above, assume that:
\begin{itemize}
\item[1)] $\dim(|D_0|)=\dim(|D_t|)$;\label{liscezza_schema_hilbert} 
\item[2)] the linear system  $W_{p,m-1}\subset |D_0|$ of curves with an $(m-1)$-tacnode
at $p\in E$ has dimension $\dim(|D_0|)-m+1$.
\end{itemize} 
Denote by $\mathcal V_{D_t,\small1^{d_2},\,2^{d_3},\ldots,\,{(m-1)}^{d_m}}^{\mathcal X_t}\subset |D_t|$ the Zariski closure of the locus in $|D_t|$ of irreducible curves with exactly $d_j$ singularities
of type $A_{j-1},$ for every $2\leq j\leq m,$  and no further singularities.
Then, for a general $t\neq 0$, there exists a
non-empty irreducible component $V_t$ of  $\mathcal V_{D_t,\small1^{d_2},\,2^{d_3},\ldots,\,{(m-1)}^{d_m}}^{\mathcal X_t}\subset |D_t|$
whose general element $[C_t]\in V_t$ is a regular point of $V_t,$  i.e. $\dim(T_{[C_t]}V_t)=\dim(T_{[C_t]}ES(C_t))=
h^0(C_t,\mathcal N^\prime_{C_t|\mathcal X_t})=\dim(|D_t|)-\sum_{j=2}^m(j-1)d_j.$
\end{theorem} 

The proof of this theorem will occupy us until Corollary \ref{ingenerale}. In the remainder of the section we will discuss several consequences
and applications of Theorem \ref{prop:defcusps}.
 
We want to obtain curves with $A_k$-singularities on the general fibre $\mathcal X_t$ of $\mathcal X$ as deformations of $C\subset\mathcal X_0.$
The moduli space of deformations of $C$ in $\mathcal X$ is contained in an irreducible component
$\mathcal H$ of the relative Hilbert scheme $\mathcal H^{\mathcal X|\mathbb A^1}$
of the family $\mathcal X\to\mathbb A^1.$
Let $\pi_\mathcal H:\mathcal H\to\mathbb A^1$ be the natural map from $\mathcal H$ to 
$\mathbb A^1.$ By the hypothesis of regularity on the fibres of the family $\mathcal X,$ 
we have that the general fiber $\mathcal H_t$ of $\pi_\mathcal H$ coincides with 
the linear system $|\mathcal O_{\mathcal X_t}(D_t)|$, whereas, in general, 
the central fiber $\mathcal H_0$ of $\pi_\mathcal H$ consists of several irreducible components 
of the Hilbert scheme of $\mathcal X_0$, only one of which, call it $\mathcal H_0^0$, can be generically identified with
$|\mathcal O_{\mathcal X_0}(D_0)|$. This happens because the limit line bundle on $\mathcal X_0$
 of a line bundle on $\mathcal X_t$
is unique only up to twisting with a multiple of $\mathcal O_{\mathcal X}(A)$ (see for example \cite[Section 2.2]{cd}).
Moreover,  by standard deformation theory 
 (cf. for example \cite[Prop.~4.4.7]{ser}), the hypothesis  $\dim(|D_0|)=\dim(|D_t|)$ ensures
 smoothness of $\mathcal H$ at the point $[C]$ corresponding to $C.$
Again, since $C$ is a local complete intersection in the smooth variety $\mathcal X$ (see \cite[Prop.~ 1.1.9]{ser}),
we have the same exact sequence introduced in the previous section
 \begin{equation}\label{sequenceone}
 \xymatrix{
0  \ar[r] & \Theta_C\ar[r] & \Theta_{\mathcal X}|_C\ar[r] ^\alpha
& \mathcal N_{C|\mathcal X}\ar[r]^\beta & T^1_C\ar[r] & 0,}
 \end{equation}
 where $T^1_C$ is the first cotangent sheaf of $C$ and where the kernel $\N^\prime_{C|\mathcal X}$
 of $\beta$  is called the equisingular normal sheaf of $C$ in $\mathcal X$, cf. \cite[Prop.~ 1.1.9]{ser}. 
 By hypothesis, $C$ is smooth outside $C\cap E=C_A\cap C_B,$ 
 it has an $m$-tacnode at $p$ and nodes at the other intersection
 points of $C_A$ and $C_B.$ 
So, at every node $r$ of $C$ we have that
 $$
 T^1_{C,r}\simeq \mathbb C[x,y]/(x,y)\simeq\mathbb C,$$ 
 while the stalk of $T^1_C$ at $p$ is given by
 $$T^1_{C,p}\simeq \mathbb C^{2m-1}\simeq\mathbb C[y,z]/J_f,$$
 where $J_f=(2y-z^m,mz^{m-1}y)$ is the Jacobian ideal of $f(y,z)=(y-z^{m})y$ \cite{diaz_harris}.
In particular, choosing 
$$
\{1,\,z,\,z^2,\,\dots,\,z^{m-1},\,y,\,yz,\,yz^2,\,\dots,\,yz^{m-2}\}$$
as a base for $T^1_{C,p}$ and using the same notation as in \cite{ch} and \cite{diaz_harris}, 
the versal deformation family $\mathcal C_p\to T^1_{C,p}$ of the singularity of $C$ at $p$ has equation
\begin{equation}\label{versal-family-m-tacnode}
\mathcal C_p: F(y,z;\underline\alpha,\underline\beta)=
y^2+\Big(\sum_{i=0}^{m-2}\alpha_iz^i+z^m\Big)y+\sum_{i=0}^{m-1}\beta_iz^i=0,
\end{equation}
while the versal family $\mathcal C_r\to T^1_{C,r}$ of the node has equation
$$
xy+t=0.
$$
Denote by $\mathcal D\to\mathcal H$ the universal family parametrized by $\mathcal H$
and by $\mathcal C_q\to T^1_{C,q}$ the versal family parametrized by $T^1_{C,q}$.
By versality, 
for every singular point $q$ of $C$ there exist analytic neighborhoods
$U_q$ of $[C]$ in $\mathcal H,$ $U_q^\prime$ of $q$ in $\mathcal D$ and  $V_q$ of $\underline 0$ in $T^1_{C,q}$
 and a map $\phi_q:U_q\to V_q$ 
such that the family
$\mathcal D|_{U_q}\cap U_q^\prime$ is isomorphic to the pull-back of $\mathcal C_q|_{V_q}$, with respect to $\phi_q$,
\begin{equation}\label{versalmap}
\xymatrix{
\mathcal C_q  \ar[d] & \mathcal C_q |_{V_q} \ar[d]  \ar@{_{(}->}[l] & 
U_q\times_{V_q}\mathcal C_q |_{V_q}  \ar[l] \ar[r]^{\hspace{0.5cm}\simeq} \ar[dr] & \mathcal D|_{U_q}\cap U_q^\prime
\ar@{^{(}->}[r]  \ar[d] & \mathcal D  \ar[d] \\
T^1_{C,q} &  V_q \ar@{_{(}->}[l]  & & \ar[ll]_{\phi_q} U_q \ar@{^{(}->}[r]  & \mathcal H.}
\end{equation}
Furthermore, by the standard identifications of the tangent space $T_{[C]}\mathcal H$ at $[C]$ to 
 the relative Hilbert scheme with $H^0(C, \mathcal N_{C|\mathcal X} )$
 and of the versal deformation space of a plane singularity with its tangent space at the origin,
 the natural map 
 $$H^0(\beta):H^0(C, \mathcal N_{C|\mathcal X})\to H^0(C,T^1_C)=\oplus_{q\in \textrm{Sing}(C)}T_{C,q}^1 $$
  induced by \eqref{sequenceone}  is identified with
 the differential $d\phi_{[C]}$ at $[C]$ of the versal map 
 $$
 \phi=\oplus_{q\in \Sing(C)}\phi_q:\cap_{q\in \Sing(C)} U_q\subset\mathcal H\to H^0(C,T^1_C).
 $$ 
We want to obtain the existence of curves with the desired singularities on $\mathcal X_t$ in $|D_t|$  by versality.
In particular, we will prove that  the locus, in the image of $\phi,$ of curves with $d_j$ singularities
of type $A_{j-1},$ for every $j$ as in the statement of Theorem \ref{prop:defcusps}, is 
non-empty. In order to do this, we observe that, no matter how we deform $C\subset\mathcal X_0$ to a curve on $\mathcal X_t,$ 
the nodes of $C$ (lying on $E$) are necessarily smoothed. Thus singularities of type $A_k$ may arise by deformation of the tacnode
of $C$ at $p$ only. For this reason we may restrict our attention to the versal map
$\phi_p$ of \eqref{versalmap}
and its differential 
 $$d\phi_p:H^0(C, \mathcal N_{C|\mathcal X})\to H^0(C,T^1_C)\to T_{C,p}^1 $$
 at the point $[C]\in\mathcal H_0^0$, where, as above,
 $\mathcal H^0_0$ is the irreducible component of the central fibre $\mathcal H_0$ of the relative Hilbert scheme $\mathcal H^{\mathcal X|\mathbb A^1}$
containing $[C]$.  

We first study the kernel of $d\phi_p$.\footnote{
The kernel and the image of $d\phi_{p}$ are also computed in  \cite[Thm.~2.3]{C}.
We give a different and more detailed proof of these two results. This will make the  
proof of Theorem \ref{thm:maink3} shorter.}
 Let $r_C:H^0(\mathcal X_0,\mathcal O_{\mathcal X_0}(D_0))\to H^0(C,\mathcal
O_{C}(D_0))=H^0(C,\mathcal N_{C|\mathcal X_0})$
be the natural restriction map, $W_{p,m-1}$ the linear system of curves in $|\mathcal O_{\mathcal X_0}(D_0)|$ with an $(m-1)$-tacnode at $p$,
as in Definition \ref{W_{p,m}}, and $\mathcal W_{p,m-1}\subset H^0(\mathcal X_0,\mathcal O_{\mathcal X_0}(D_0))$ the vector space
such that $\mathbb P(\mathcal W_{p,m-1})=W_{p,m-1}$.

\begin{lemma}\label{lemma-ker}
We have 
 \begin{equation}\label{kerneldifferenziale}
\ker(d\phi_p)=\ker(d\phi_{[C]})=H^0(C,\mathcal N^\prime_{C|\mathcal X})=r_C(\mathcal W_{p,m-1})=H^0(C,\mathcal N^\prime_{C|\mathcal X_0}).
\end{equation}
More generally, let $C^\prime=C^\prime_A\cup C^\prime_B\in |D_0|$ be any reduced curve and 
$x\in C^\prime\cap E$ a singular point of $C^\prime$ on $E$. We have that,
if $K_x$ is the kernel of the natural map $H^0(C^\prime,\mathcal N_{C^\prime|\mathcal X})\to T^1_{C^\prime,x}$, then 
\begin{equation}\label{kerneldifferenziale2}\small
 H^0(C^\prime,\mathcal N^\prime_{C^\prime|\mathcal X})\subseteq K_x\subseteq H^0(C^\prime,\mathcal N_{C^\prime|\mathcal X_0})
\,\, {\rm{and}}\,\, H^0(C^\prime,\mathcal N^\prime_{C^\prime|\mathcal X})=H^0(C^\prime,\mathcal N^\prime_{C^\prime|\mathcal X_0}).
\end{equation}
\normalsize
Finally, using the notation above, if $C^\prime$ has an $m$-tacnode at $x$, then 
\begin{equation}\label{ker2} 
K_x\subseteq r_{C^\prime}(\mathcal W_{x,m-1}),\,\,{\rm{with\,equality\,if}}\,\,\dim(W_{x,m-1})=\dim(|D_0|)-m+1.
\end{equation}
\end{lemma} 

\begin{proof}
From what we observed above, we have that
$\ker(d\phi_{[C]})=\ker(H^0(\beta))=H^0(C,\mathcal N^\prime_{C|\mathcal X}),$ where $\N^\prime_{C|\mathcal X}$ is the
equisingular normal sheaf of $C$ in $\mathcal X$. Moreover we have the inclusion  $\ker(d\phi_{[C]})
\subseteq\ker(d\phi_p)$. We want to prove that equality holds and that 
$H^0(C,\mathcal N^\prime_{C|\mathcal X})=r_C(\mathcal W_{p,m-1})=H^0(C,\mathcal N^\prime_{C|\mathcal X_0}).$

Consider the localized exact sequence
 \begin{equation}\label{localizationsequenceone}
 \xymatrix{
 0\ar[r] & \mathcal N^\prime_{C|\mathcal X,\,p}\ar[r] &
\mathcal N_{C|\mathcal X,\,p}\ar[r] &  T^1_{C,p}\ar[r] & 0.}
 \end{equation} 
 Using local analytic coordinates $x,\,y,\,z,\,t$ at $p$ as in \eqref{tacnode}, we may identify:
 \begin{itemize}
 \item the local ring $\mathcal O_{C,\,p}=\mathcal O_{\mathcal X,p}/\mathcal I_{C|\mathcal X,p}$ of $C$ at $p$ 
 with $\mathbb C[x,y,z]/(f_1,f_2)$, where $f_1(x,y,z)=x+y+z^m$ 
 and $f_2(x,y,z)=xy,$ \\
 \item the $\mathcal O_{C,p}$-module $\mathcal N_{C|\mathcal X,\,p}$ with the free $\mathcal O_{\mathcal X,\,p}$-module 
 $\mathfrak{hom}_{\mathcal O_{\mathcal X,\,p}}(\mathcal I_{C|\mathcal X ,\,p}, \mathcal O_{C,p}),$ generated by the
 morphisms $f_1^*$ and $f_2^*$,  defined by 
 $$f_i^*(s_1(x,y,z)f_1(x,y,z)+s_2(x,y,z)f_2(x,y,z))=s_i(x,y,z),\,\mbox{for}\, i=1,\,2$$
 \end{itemize}
 and, finally,
 \begin{itemize}
 \item the $\mathcal O_{C,p}$-module
 \begin{eqnarray*}
 (\Theta_{\mathcal X}|_C)_{\, p}&\simeq &\Theta_{\mathcal X,p}/(I_{C,p}\otimes\Theta_{\mathcal X,p})\\
& \simeq & \langle \partial /{\partial x},\partial /{\partial y},\partial / {\partial z},\partial /{\partial t} \rangle
 _{\mathcal O_{C,\,p}}/
 \langle {\partial}/{\partial t}-x\partial /{\partial y}-y\partial /{\partial x}\rangle
 \end{eqnarray*}
 with the free $\mathcal O_{\mathcal X,\,p}$-module generated by the derivatives 
 $\partial /{\partial x},\partial / {\partial y},\partial /{\partial z}.$ 
 \end{itemize}
With these identifications, 
the localization $\alpha_p:(\Theta_{\mathcal X}|_C)_{\, p}\rightarrow \mathcal N_{C|\mathcal X,\,p}$ of the sheaf map $\alpha$ from \eqref{sequenceone} is defined by
 \begin{eqnarray*}
\alpha_p(\partial /{\partial x})&=&\,\,\,\,\,\,\Big(s=s_1f_1+s_2f_2\rightarrowtail 
\partial s/\partial x=_{\mathcal O_{C,p}}
s_1\partial f_1 /{\partial x}
+s_2\partial f_2 /{\partial x}\Big)\\
&=&\,\,\,\,\,\, f_1^*+yf_2^*,\\
\alpha_p(\partial /{\partial y})&=&\,\,\,\,\,\,f^*_1+xf^*_2\,\,\,\textrm{and}\\
\alpha_p(\partial /{\partial z})&=&\,\,\,\,\,\,mz^{m-1}f^*_1.
\end{eqnarray*}
By definition of $\mathcal N^\prime_{C|\mathcal X}$, a section $s\in \mathcal N_{C|\mathcal X,\,p}$ 
is equisingular at $p$, i.e. $s\in \mathcal N^\prime_{C|\mathcal X,\,p}$, if and only if there exists a section
$$u=u_x(x,y,z)\partial /{\partial x}+u_y(x,y,z)\partial /{\partial y}+u_z(x,y,z)\partial /{\partial z}\in {\Theta_{\mathcal X}|_C}_{\, p}$$
such that $s=\alpha_p(u)$. Hence, locally at $p$, first order equisingular deformations of $C$ in $\mathcal X$ have equations 
\begin{eqnarray}\label{equisingulardeformationtacnode}
\left\{\begin{array}{l}
x+y+z^m+\epsilon (u_x+u_y+mz^{m-1}u_z)=0\\
xy+\epsilon (yu_x+xu_y)=0.
\end{array}\right.
\end{eqnarray}
The first equation above gives an infinitesimal deformation of the Cartier divisor 
 cutting $C$ on $\mathcal X_0$, while the second equation 
 gives an infinitesimal 
 deformation of $\mathcal X_0$ in $\mathcal X.$ More precisely,
by \cite[Section 2]{chm},  the equation $xy+\epsilon (yu_x+xu_y)=0$ is the local equation at $p$ of an equisingular deformation 
of  $\mathcal X_0$ in $\mathcal X$ preserving the singular locus $E$. But 
$\mathcal X_0$ may be deformed in $\mathcal X$ only to a fiber and $\mathcal X_0$ is the only 
singular fiber of $\mathcal X.$ It follows that
the polynomial  $yu_x(x,y,z)+xu_y(x,y,z)$ in the second equation of  
\eqref{equisingulardeformationtacnode}  must be identically zero. In particular,
we obtain that 
$$
\ker(d\phi_p)\subset H^0(C,\mathcal N_{C|\mathcal X_0}).
$$ 
Since every infinitesimal deformation of $C$ in $\mathcal X_0$ preserves the nodes of $C$
lying on $E$, i.e. all nodes of $C$, we deduce that  $\ker(d\phi_p)\subseteq \ker(d\phi_{[C]})$
and thus
$\ker(d\phi_p)=\ker(d\phi_{[C]})=H^0(C,\mathcal N^\prime_{C|\mathcal X})\subset 
H^0(C,\mathcal N_{C|\mathcal X_0})$. Moreover, since the natural linear map $H^0(C,\mathcal N_{C|\mathcal X_0})\to H^0(T^1_C)$
is the restriction of  $d\phi_{[C]}$ to $H^0(C,\mathcal N_{C|\mathcal X_0})$, with kernel $H^0(C,\mathcal N^\prime_{C|\mathcal X_0})$, 
we also obtain the equality $H^0(C,\mathcal N^\prime_{C|\mathcal X_0})=H^0(C,\mathcal N^\prime_{C|\mathcal X})$.
This in particular is consistent with the very well-known fact that there do not exist deformations 
\begin{equation} \label{versalmap2}
\xymatrix{
 C  \hspace{0.3cm} \subset \hspace{-0.7cm} \ar[d]     & \mathcal C \hspace{0.3cm} 
 \subset \hspace{-0.7cm} \ar[d]  &  \mathcal X \ar[dl] \\
\underline 0 \hspace{0.4cm} \in \hspace{-0.8cm} & \mathbb A^1 &  }
\end{equation}
of $C$ in $\mathcal X$ preserving the nodes of $C$, except for deformations of $ C$ in $\mathcal X_0$
(see \cite[Section 2]{galati} for a proof).

Notice finally that, by the argument above, the inclusions $$H^0(C,\mathcal N^\prime_{C|\mathcal X})\subset 
\ker(d\phi_p)\subset H^0(C,\mathcal N_{C|\mathcal X_0})$$
hold independently of the kind of singularities of $C$ on $E$. This proves \eqref{kerneldifferenziale2}.

Now it remains to show that $H^0(C,\mathcal N^\prime_{C|\mathcal X})=r_C(\mathcal W_{p,m-1}).$
Consider the first equation of \eqref{equisingulardeformationtacnode}.
By the fact that the polynomial $yu_x(x,y,z)+xu_y(x,y,z)$ is identically zero, we deduce that
\begin{itemize} 
\item $u_x(0,0,0)=u_y(0,0,0)=0$ and
\item for every $n\geq 1$, no $z^n$-terms appear in $u_x(x,y,z)$ and $u_y(x,y,z),$ 
no $y^n$-terms and $y^nz^m$-terms appear in $u_x(x,y,z)$ and, 
finally, no $x^n$-terms and $x^nz^m$-terms appear in $u_y(x,y,z).$
\end{itemize}
In particular, local equations at $p$ on $B$ of equisingular infinitesimal deformations of $C$ are given by
\begin{eqnarray}\label{equisingulardeformationtacnodeonA}
\left\{\begin{array}{l}
y q(y,z)+z^m+\epsilon m z^{m-1}u_z(x,y,z)=0\\
x=0,
\end{array}\right.
\end{eqnarray}
where $q(y,z)$ is a polynomial with variables $y$ and $z,$ and similarly on $A.$ This proves that
$H^0(C, \mathcal N^\prime_{C|\mathcal X})\subset r_C(\mathcal W_{p,m-1})$ and more generally 	the inclusion in \eqref{ker2}.
The opposite inclusion $r_C(\mathcal W_{p,m-1})\subset  H^0(C, \mathcal N^\prime_{C|\mathcal X})$  follows from a naive dimensional 
computation. Indeed, by the hypothesis 2), 
$\dim(r_C(\mathcal W_{p,m-1}))=\dim(W_{p,m-1})=\dim(|D_0|)-m+1.$ On the other hand, 
if $W\subset  W_{E,m}$ is the irreducible component containing the point $[C]$, then $\dim(W)\geq \dim(|D_0|)-m+1$
and $W$ is contained in the Zariski closure of the family of locally trivial deformations of $C$.
Thus $W$ is the Zariski closure of the locus of locally trivial deformations of $C$. Its tangent space at $[C]$ is isomorphic 
to $H^0(C, \mathcal N^\prime_{C|\mathcal X})$ and \eqref{kerneldifferenziale} is proved. The same argument shows that
$K_x=r_{C^\prime}(\mathcal W_{x,m-1})$ in \eqref{ker2} if $\dim(W_{x,m-1})=\dim(|D_0|)-m+1.$ The lemma is proved.
  \end{proof}  
 
 We now describe the image of $d\phi_{p}.$
  
\begin{lemma}\label{lemma-im} 
Let $\alpha_0,...,\alpha_{m-2},\beta_0,...,\beta_{m-1}$ be 
coordinates on $T^1_{C,p}$ as above. Then the image  $H_p\subset T^1_{C,p}$  of  $d\phi_p$ is given by the equations
\begin{equation}\label{equazioniH_p}
H_p=d\phi_p(H^0(C,\mathcal N_{C|\mathcal X})):\beta_1=\cdots=\beta_{m-1}=0.
\end{equation}
Moreover, the image of $H^0(C,\mathcal N_{C|\mathcal X_0})$ under $d\phi_p$ is the linear subspace
of $H_p$ given by 
\begin{equation}\label{equazioniH_p2}
d\phi_p(H^0(C,\mathcal N_{C|\mathcal X_0}))=T_0V_{1^m}:\beta_0=\beta_1=\cdots=\beta_{m-1}=0
\end{equation}
and coincides with the tangent space at $0$ of the locus $V_{1^m}$ of $m$-nodal curves.
 \end{lemma}
 
 From the following remark to the end of the section,  $V_p\subset T_C^1$ and $U_p\subset \mathcal H_{\mathcal X|\mathbb A^1}$
 are analytic neighborhood as in diagram \eqref{versalmap}. 
 
 \begin{remark}\label{remark-im} 
 Notice that, because of the choice we have made for the local analytic equation of $C$ at $p$, we have that  the locus $V_{1^m}\subset T^1_{C,p}$ 
 of $m$-nodal curves is the linear space defined by the equations $\beta_0=\beta_1=\cdots=\beta_{m-1}=0$, coinciding with its tangent space at $0$.
 In general, the locus $V_{1^m}$ of $m$-nodal curves in the versal deformation space of the $m$-tacnode $T^1_{C,p}$ is a smooth variety, parametrizing equigeneric deformations.  Lemma \ref{lemma-im}  shows in general that, if $C$ is as in Theorem \ref{prop:defcusps}, then
 $d\phi_p(H^0(C,\mathcal N_{C|\mathcal X_0}))=T_0V_{1^m}$ and consequently
 $\phi_p(\mathcal H_0^0\cap U_p)=V_{1^m}\cap V_p$.
 \end{remark}

\begin{proof}[Proof of Lemma \ref{lemma-im}.]
We first observe that $\dim(H_p)=m$. Indeed, by Lemma \ref{lemma-ker} and hypothesis $2)$ of Theorem \ref{prop:defcusps},
we have that
 \begin{eqnarray*}
 \dim(H_p) &=& h^0(C, \mathcal N_{C|\mathcal X})-\dim(r_C(\mathcal W_{p,m-1}))=m.
\end{eqnarray*}
Moreover, since $\ker(d\phi_p)\subset H^0(C, \mathcal N_{C|\mathcal X_0})$ and $H^0(C, \mathcal N_{C|\mathcal X_0})$ has
codimension $1$ in $H^0(C, \mathcal N_{C|\mathcal X})$, we have that $H_p$ contains $d\phi_p(H^0(C, \mathcal N_{C|\mathcal X_0}))$
 as a codimension-$1$ linear subspace. 
Again by the hypothesis $2)$ of Theorem \ref{prop:defcusps} and Remark \ref{transversality}, the $(m-1)$-linear space
$d\phi_p(H^0(C,\mathcal N_{C|\mathcal X_0}))$ is contained in the tangent 
space $T_{\underline 0}V_{1^m}\subset T^1_{C,p}$ at $\underline 0$ to the Zariski
 closure  $V_{1^m}\subset T^1_{C,p}$ of the locus of $m$-nodal curves. Now
it is easy to verify that, using the coordinates \eqref{versal-family-m-tacnode} on  $T^1_{C,p}$, 
the equations of $V_{1^m}$ are
$
\beta_0=\cdots=\beta_{m-1}=0.
$
Hence
\begin{equation} \label{beta0}
d\phi_{p}(H^0(C,\mathcal N_{C|\mathcal X_0}))=
T_{\underline 0}V_{1^m}=V_{1^m}=\phi_p(\mathcal H_0^0\cap U_p)
\end{equation}
(cf. Remark \ref{remark-im}). By \eqref{beta0}, in order to find a base of $H_p$, it is enough to find the image by
$d\phi_p$ of the infinitesimal deformation $\sigma\in H^0(C,\mathcal N_{C|\mathcal X})\setminus
H^0(C,\mathcal N_{C|\mathcal X_0})$ having equations
\begin{eqnarray}\label{sigma}
\left\{\begin{array}{l}
x+y+z^m=0\\
xy=\epsilon.
\end{array}\right.
\end{eqnarray}  
The image of $\sigma$ is trivially the point corresponding to the curve $y(y+z^m)=\epsilon.$
Thus, the equations of $H_p\subset T^1_{C,p}$ are given by \eqref{equazioniH_p}.
\end{proof}

Now let $d_2,\ldots,d_m$ be non-negative integers such that 
$ m=\sum_{j=2}^m(j-1)d_j\,+1,$
as in the statement of Theorem \ref{prop:defcusps}. Let $V_p\subset T^1_{C,p}$ be the analytic open set as in \eqref{versalmap} and denote by 
$$V_{\small 1^{d_2},\,2^{d_3},\ldots,\,{(m-1)}^{d_m}}\subset V_p\subset T^1_{C,p}$$
the Zariski closure of the locus in $V_p$ of curves with exactly $d_j$ 
singularities of type $A_j$, for every $j,$ and no further singularities. The following proposition
implies Theorem \ref{prop:defcusps}, as indicated below.

\begin{proposition}\label{referee}
We have 
\begin{equation}\label{intersection}
V_{\small 1^{d_2},\,2^{d_3},\ldots,\,{(m-1)}^{d_m}}\cap H_p=B_1 \cup B_2 \cup \cdots \cup B_k,
\end{equation}
where every $B_i$ is an irreducible and reduced affine curve containing $0$, 
whose general element corresponds to 
a curve with exactly $d_j$ singularities of type $A_{j-1}$, for every $j,$ and no further singularities.
\end{proposition}

\begin{remark}\label{k}
The number $k$ is explicitly determined by a combinatorial argument at the end of the proof of Claim \ref{referee2}.
\end{remark}

\begin{proof}[Proof of  Theorem \ref{prop:defcusps}]
By Lemmas \ref{lemma-ker} and \ref{lemma-im},
we know that the image $\phi_p(U_p)\subset V_p\subset T^1_{C,p}$ of $ U_p\subset\mathcal H$ by
 $\phi_p$ is an $m$-dimensional subvariety of $V_p$ that is 
smooth at $\underline 0$.
By Remark \ref{A_k} and the openness of versality (more precisely, by the properties \cite[(3.5), (3.6)]{diaz_harris} of versal deformation families), 
up to restricting $V_p$, we have that 
$V_{\small 1^{d_2},\,2^{d_3},\ldots,\,{(m-1)}^{d_m}}\subset V_p$ is a variety of pure codimension
$\sum_j(j-1)d_j=m-1,$ which is non-empty by \eqref{intersection}. Moreover, up to restricting $V_p$ again, we may assume that
$\underline 0$ is contained in every irreducible component of $V_{\small 1^{d_2},\,2^{d_3},\ldots,\,{(m-1)}^{d_m}}$. 
Hence $V_{\small 1^{d_2},\,2^{d_3},\ldots,\,{(m-1)}^{d_m}}\cap\phi_p(U_p)$
is non-empty and each of its irreducible components has dimension $\geq 2m-1-2m+2=1.$ 
By recalling that $H_p=d\phi_p(H^0(C,\N_{C|\mathcal X}))=T_{\underline 0}\phi_p(U_p)$
and $\phi_p(U_p)$ is smooth at $\underline 0,$ we see that \eqref{intersection}
implies that the intersection
$$V_{\small 1^{d_2},\,2^{d_3},\ldots,\,{(m-1)}^{d_m}}\cap\phi_p(U_p)
=B^\prime_1 \cup B^\prime_2 \cup \cdots \cup B^\prime_k,
$$ has pure dimension $1.$
Now notice that, by semicontinuity, since $\phi_p^{-1}(\underline 0)$ is smooth of (maximal) codimension $m$ in 
the relative Hilbert scheme $\mathcal H_{\mathcal X|\mathbb A^1},$ the same is true for the fibre $\phi_p^{-1}(x)$,
for $x\in V_p$, up to restricting $V_p$. 
More precisely, again by the properties \cite[(3.5), (3.6)]{diaz_harris} of versal deformation families, if $x\in B_i^\prime$, with $x\neq \underline 0$,
 is a point sufficiently close to $\underline 0$ of any irreducible component of 
 $V_{\small 1^{d_2},\,2^{d_3},\ldots,\,{(m-1)}^{d_m}}\cap\phi_p(U_p)$, and
 $[C_t]\in\phi_p^{-1}(B_i^\prime)\cap\mathcal H_t$, then we have that
$T_{[C_t]}(\phi_p^{-1}(B_i^\prime))\simeq H^0(C_t,\mathcal N_{C_t|\mathcal X}^\prime)$ and thus 
$\phi_p^{-1}(B_i^\prime\setminus\underline 0)$ is the equisingular deformation locus of $C_t$
in $\mathcal X$ {\it scheme theoretically}, cf. Definition \ref{definition-regularity} and Remark \ref{A_k}.
In particular, $\phi_p^{-1}(B_i^\prime\setminus\underline 0)\cap \mathcal H_t$ is the equisingular deformation locus
of $C_t$ in $\mathcal X_t$,
$$
T_{[C_t]}(\phi^{-1}_p (B_i^\prime)\cap \mathcal H_t)\simeq H^0(C_t,\mathcal N_{C_t|\mathcal X_t}^\prime),
$$
and 
$$
h^0(C_t,\mathcal N_{C_t|\mathcal X_t}^\prime)=h^0(C_t,\mathcal N_{C_t|\mathcal X}^\prime)-h^0(\mathcal X,\mathcal O_{\mathcal X})=
\dim(|D_t|)-m+1.$$
This proves that $[C_t]$ is a regular point of an irreducible component $V$ of the variety
$\mathcal V_{D_t,\small 1^{d_2},\,2^{d_3},\ldots,\,{(m-1)}^{d_m}}^{\mathcal X_t}$. 
\end{proof}

\begin{proof}[Proof of  Proposition \ref{referee}]
We use the same approach as Caporaso and 
Harris in  \cite[Section 2.4]{ch}. Equation \eqref{versal-family-m-tacnode},
defining the versal family $\mathcal C_p$ of the $m$-tacnode, has degree $2$ in $y.$ In particular,
for every point $x=(\underline\alpha,\underline\beta)\in T^1_{C,p},$ the corresponding curve 
$\mathcal C_{p,x}:F(y,z;\underline\alpha,\underline\beta)=0$ is
a double cover of the $z$-axis.  Moreover, the discriminant map
$$
\Delta:T^1_{C,p}\to\mathcal P^{2m-1}=\Big\{z^{2m}+\sum_{i=0}^{2m-2}a_iz^i, a_i\in \mathbb C\Big\}    
$$
from $T^1_{C,p}$ to the affine space of monic polynomials of degree $2m$ with no 
$2m-1$ term, defined by
$$
\Delta(\underline\alpha,\underline\beta)(z)=\Big(\sum_{i=0}^{m-2}\alpha_iz^i+z^m\Big)^2
-4\Big(\sum_{i=0}^{m-1}\beta_iz^i\Big),
$$ is an 
isomorphism (see \cite[p. 179]{ch}). Thus, we may study curves in the versal deformation family
of the $m$-tacnode in terms of 
the associated discriminant polynomial. In particular, for every point $(\underline\alpha,\beta_0,0,\ldots,0)\in H_p,$
which we will shortly denote by  $(\underline\alpha,\beta_0):=(\underline\alpha,\beta_0,0,\ldots,0),$
the corresponding discriminant polynomial is given by
\begin{equation}\label{discriminant_on_H_p}
\small\Delta(\underline\alpha,\beta_0)(z) =\Big(\sum_{i=0}^{m-2}\alpha_i z^i+z^m\Big)^2-4\beta_0
\small = \Big( \nu(z)-2\sqrt{\beta_0} \Big) \Big(\nu(z)+2\sqrt{\beta_0} \Big),
\end{equation}
where we set 
$
\nu(z):=\sum_{i=0}^{m-2}\alpha_iz^i+z^m.
$
From now on, since we are interested in deformations of the $m$-tacnodal curve
 $C\subset \mathcal X_0$ to curves on $\mathcal X_t,$ with $t\neq 0,$ we will always assume 
 $\beta_0\neq 0,$  being $V_{1^m}:\beta_0=\beta_1=\cdots=\beta_{m-1}=0$ the image in $T^1_{C,p}$ of 
 infinitesimal (and also effective) deformations of $C$ in $\mathcal X_0.$ Writing down explicitly 
the derivatives of the polynomial $F(y,z;\underline\alpha,\underline\beta)$, 
one may  verify that a point $x=(\underline\alpha,\beta_0)$ parametrizes a curve $\mathcal C_{p,x}$ with an $A_k$-singularity
at the point $(z_0,y_0)$ if and only if $z_0$ is a root of multiplicity $k+1$ of the discriminant polynomial 
$\Delta(\underline\alpha,\beta_0)(z).$ Our existence problem is thus equivalent to the following.

\begin{claim}\label{referee2}
Let $d_2,\ldots,d_m$ be an $(m-1)$-tuple of non-negative integers  satisfying \eqref{m-1}. The locus of points $(\underline\alpha,\beta_0)\in H_p$
such that the discriminant polynomial $\Delta(\underline\alpha,\beta_0)(z)$ has exactly $d_j$ roots
of multiplicity $j,$ for every $2\leq j\leq m,$ and no further multiple roots, is non-empty of pure dimension $1$. 
\end{claim} 
The claim will  be proved right below.
\end{proof}

To prove the last claim we need the following auxiliary result, 
whose proof is postponed to Appendix \ref{tricicli_proof}. 

\begin{lemma} \label{tricicli}
  Let $m\geq n \geq 2$ be integers and $(d_2^+,d_2^-,\ldots, d_n^+,d_n^-)$ be a 
  $2(n-1)$-tuple of non-negative integers satisfying
  \begin{eqnarray}
 m&=& \displaystyle\sum_{j=2}^{n}(j-1)(d_j^++d_j^-)+1 \geq 2\,\,\,\textrm{and} \label{eq:amm1}
\\   m&\geq & \displaystyle\sum_{j=2}^{n}jd_j^{\pm}>0. \label{eq:amm2}
  \end{eqnarray}
  
Then there exists a triple of 
permutations $(\tau^+,\tau^-,\sigma)$ of $m$ indices, such that 
$\tau^{\pm}$ has cyclic structure $\Pi_{j=2}^n j^{d_j^{\pm}},$ 
 $\sigma$ is a cycle of order $m$ and
$\sigma \tau^+ \tau^-=1$.
\end{lemma}
\begin{definition}\label{admissible}
An {\rm admissible  $2(n-1)$-tuple} is a $2(n-1)$-tuple of non-negative integers $(d_2^+,d_2^-,\ldots, d_n^+,d_n^-)$ 
satisfying \eqref{eq:amm1} and \eqref{eq:amm2}. 
\end{definition}

\begin{proof}[Proof of  Claim \ref{referee2}]
Let $d_2,\ldots,d_m$ be an $(m-1)$-tuple of non-negative integers  satisfying \eqref{m-1}. 
Assume there exists a point $(\underline\alpha,\beta_0)\in H_p$ such that the discriminant polynomial 
$\Delta(\underline\alpha,\beta_0)(z)$ has the desired properties.
Then, since $\nu(z)-2\sqrt{\beta_0}$ and $\nu(z)+2\sqrt{\beta_0}$
cannot have common factors for $\beta_0\neq 0,$ there exist non-negative integers $d_j^+,\,d_j^-$ such
that $d_j=d_j^++d_j^-$ and the $d_j$ roots of multiplicity $j$
of the discriminant $\Delta(\underline\alpha,\beta_0)(z)$ are distributed as $
d_j^{+}$  roots of $ \nu(z) + 2\sqrt{\beta_0}$ and $d_j^-$ roots of $ \nu(z) - 2\sqrt{\beta_0}.$
The obtained $2(m-1)$-tuple of non-negative integers $(d_2^+,d_2^-,\ldots,d_m^+,d_m^-)$ is admissible (see
Definition \ref{admissible}).
The polynomial $\nu(z)=z^m+\sum_{i=0}^{m-2}\alpha_iz^i$ defines a degree $m$ covering
 $\nu:\mathbb P^1\to\mathbb P^1,$ totally ramified at $\infty$ and with further $d_j^\pm$ ramification 
points of order $j-1$ over $\pm 2\sqrt{\beta_0},$ for every $2\leq j\leq m.$  
We get in this way
$$(m-1)+\sum_{j=2}^{m}(j-1)d_j^++\sum_{j=2}^{m}(j-1)d_j^-=2m-2$$ ramification points of $\nu.$
Hence $\nu$ has no further ramification by the Riemann-Hurwitz formula.  The branch points of $\nu$     
are three, consisting of $\infty$, $-2\sqrt{\beta_0}$ and $2\sqrt{\beta_0},$ if both sums $\sum_jd_j^+$
and $\sum_jd_j^-$ are positive, while the branch points of $\nu$ are only two if  
$\sum_jd_j^+=0$ or $\sum_jd_j^-=0.$
 
Consider first the case that $\nu$ has only two ramification
points, say $\infty$ and $2\sqrt{\beta_0}.$ Then we have that $\sum_jd_j^+=0$ and $d_j=d_j^-,$ 
for every $j.$ In particular, using the conditions \eqref{eq:amm1} and \eqref{eq:amm2}, we find that
$d_j=0$ for $j\neq m$ and $d_m=1.$ It follows that $\nu(z)- 2\sqrt{\beta_0}=(z-\lambda)^m,$
for a certain $\lambda.$ But the only $\lambda$ such that $(z-\lambda)^m$ has 
no 
degree $m-1$ term, is $\lambda=0.$ Thus, we get
$\nu(z)=z^m+2\sqrt{\beta_0}.$ On the other hand, for every fixed $\beta_0\neq 0,$ if  $\nu(z)=z^m+2\sqrt{\beta_0}$, then 
the associated discriminant $\Delta(\underline\alpha,\beta_0)=z^m(z^m+4\sqrt{\beta_0})$ has a root
of multiplicity $m$ and no further multiple roots.  {\it This proves the claim under the hypothesis $d_m=1$ and $d_j=0$ for $j\neq m$.
More precisely, this shows that the Zariski closure of the locus in 
$H_p\subset T^1_{C,p}$ of curves with an $A_{m-1}$-singularity is the smooth curve given by the equations}
\begin{equation}\label{A_{m-1}}
V_{\small (m-1)^{1}}\cap H_p:\begin{cases}
\begin{matrix}
\alpha_i &=& 0,&\mbox{for every}&1\leq i\leq m-2,\\
\alpha_0^2 &=& 4\beta_0. && 
\end{matrix}
\end{cases}
\end{equation}

Now consider the general case, i.e. assume that $\sum_jd_j^\pm>0.$ Then the polynomial $\nu(z)$ defines an 
$m$-covering $\nu:\mathbb P^1\to\mathbb P^1$ having branch points at $\infty,$  $-2\sqrt{\beta_0}$ and $+2\sqrt{\beta_0}$ 
with monodromy permutations $\sigma,$ $\tau^+$ and $\tau^-$ respectively, where $\sigma$ is an $m$-cycle while $\tau^\pm$ 
has cyclic structure $\Pi_{j=2}^m j^{d_j^{\pm}}.$ Moreover, the group $\langle\sigma,\,\tau^+,\,\tau^-\rangle$ is trivially transitive and, by the theory of 
coverings of $\mathbb P^1,$ we have that $\sigma\tau^+\tau^-=1.$ In fact, the class $\{\nu\Phi|\Phi:\mathbb P^1\to\mathbb P^1\,\textrm{automorphism}\,\}$ of degree $m$ coverings isomorphic to $\nu$
and with branch locus at $\infty,$  $-2\sqrt{\beta_0}$ and $+2\sqrt{\beta_0},$  is uniquely determined by the conjugacy class of the triple of permutations $(\tau^+,\tau^-,\sigma)$ (cf. e.g. \cite[Cor.~III.4.10]{mir}). Hence, for every $\beta_0\neq 0$ and for every fixed $(2m-2)$-tuple $(d_2^+,d_2^-,...,d_m^+,d_m^-)$, there exist at most finitely many polynomials $\nu(z)$ with the properties above, up to a change of variable. Moreover, a change of variable that transforms $\nu(z)$ in a polynomial with the same properties corresponds to an automorphism $\Phi: \PP^1 \to \PP^1$ such that  $\Phi(z)=az+b$ with $a^m=1$ and $b=0$, as  $\nu(z)$ is monic and has no $z^{m-1}$-term.

{\it We may now prove the non-emptiness in the general case}. Let $d_2,\ldots,d_m$  be non-negative integers satisfying \eqref{m-1}.
Then, no matter how we choose non-negative integers $d_i^{\pm}$, with  $d_i=d_i^+ +d_i^-$, 
we have that the $(2m-2)$-tuple
$\underline d= (d_2^+,d_2^-,\ldots, d_m^+,d_m^-)$ is admissible.  Furthermore, by Lemma \ref{tricicli}, there exist 
finitely many  triples of permutations $(\tau^+,\tau^-,\sigma)$ of $m$ indices such that 
$\tau^{\pm}$ has cyclic structure $\Pi_{j=2}^m j^{d_j^{\pm}},$ $\sigma$ is a cycle of order $m$ and  $\sigma \tau^+ \tau^-=1$.
Since the group $\langle\sigma,\tau^+,\tau^-\rangle$ is trivially transitive, again by the general theory of coverings of $\mathbb P^1,$ 
for every $\gamma\in\mathbb C$ there exists an $m$-covering $\nu_{(\underline d, (\tau^+,\tau^-,\sigma), \gamma)}:\mathbb P^1\to\mathbb P^1$ with branch points $\infty, -\gamma,+\gamma$ and monodromy
permutations $\sigma, \tau^+, \tau^-,$ respectively. Up to a change of variables in the domain, we may always assume
that $\nu$ is defined by a monic polynomial $\nu_{(\underline d, (\tau^+,\tau^-,\sigma), \gamma)}(z)$ with no $z^{m-1}$-term.
{\it This proves that the locus of points $(\underline\alpha,\beta_0)\in H_p$ such that the discriminant polynomial 
$\Delta(\underline\alpha,\beta_0)(z)$ has the desired properties is non-empty. We finally want to see that it has pure dimension $1$,
by writing its parametric equations explicitly. } We notice that, from what we observed above, the polynomials,
 $$ \nu_{(\underline d, (\tau^+,\tau^-,\sigma), \gamma)}(\frac{ z}{\zeta}),\,\, \rm{where}\,\,\zeta\,\,\rm{is\,an}\,\,m\rm{-th\, root\, of\, unity},$$
  are all polynomials with no $z^{m-1}$-term, and define
 an $m$-covering of $\mathbb P^1$ isomorphic to 
 $\nu_{(\underline d, (\tau^+,\tau^-,\sigma), \gamma)}$ (whose monodromy is conjugated with $(\tau^+,\tau^-,\sigma)$). More generally, if 
\begin{equation}\label{gamma=1}
\nu_{(\underline d, (\tau^+,\tau^-,\sigma), 1)}(z) =  z^m+\sum_{i=0}^{m-2}c_{i,(\underline d, (\tau^+,\tau^-,\sigma), 1)}z^{i},
\end{equation}
then, for every $\gamma\in\mathbb C\setminus 0,$   the polynomials

\begin{equation}\label{utile-parita}
\nu(z):=u^m\nu_{(\underline d, (\tau^+,\tau^-,\sigma), 1)}({z\over u})=z^m+\sum_{i=0}^{m-2}\alpha_iz^i,\,\,u^m=\gamma
\end{equation} 
are all the monic polynomials without $z^{m-1}$-term such that  $ \nu^2(z)-4\gamma^2$
has the same kind of multiple roots with the desired distribution, and defining an $m$-covering of $\mathbb P^1$
with monodromy in the conjugacy class of $(\tau^+,\tau^-,\sigma)$.
We have thus proved that the reduced and irreducible curve of parametric equations
\begin{equation}\label{parametric}
 \alpha_i= u^{m-i}c_{i,(\underline d, (\tau^+,\tau^-,\sigma), 1)},\,\mbox{for}\, \, i=0,1,\ldots,m-2,\, \,\, \mbox{and} \, \,\beta_0=\frac{u^{2m}}{4},
 \,\,u\in\mathbb C, 
\end{equation}
is an irreducible  component of $V_{\small 1^{d_2},\,2^{d_3},\ldots,\,{(m-1)}^{d_m}}\cap H_p.$
In particular, 
$$
V_{\small 1^{d_2},\,2^{d_3},\ldots,\,{(m-1)}^{d_m}}\cap H_p=B_1\cup\cdots\cup B_k
$$
is a reduced curve with $k$ irreducible components, where $k$ is the number of pairs $(\underline d, [\tau^+,\tau^-,\sigma])$, 
such that $\underline d= (d_2^+,d_2^-,\ldots, d_m^+,d_m^-)$ is an admissible $(2m-2)$-tuple  with $d_j^++d_j^-=d_j,$  for every $j,$
and
 $
[\tau^+,\tau^-,\sigma]=\{(\rho\tau^+\rho^{-1},\rho\tau^-\rho^{-1},\rho\sigma\rho^{-1})\,|\,\rho\in\mathfrak{S}_{m}  \},
$
where $\mathfrak{S}_{m}$ denotes the symmetric group of order $m$ and $\sigma$, $\tau^+$ and $\tau^-\in\mathfrak{S}_{m}$ are permutations
such that  $\sigma$ is an $m$ cycle, $\tau^\pm$ has cyclic structure $\Pi_{j=2}^m j^{d_j^{\pm}}$ and $\sigma\tau^+\tau^-=1.$
\end{proof}

Theorem \ref{prop:defcusps} is a local result. It describes all possible deformations 
of an $m$-tacnode of a reduced curve $C^\prime\in |D_0|$ at a double point $x$ of $\mathcal X_0$,
as precisely stated in the following.

\begin{corollary}\label{ingenerale}
Let $ C^\prime\subset |D_0|$ be any reduced curve with an $m$-tacnode  at a
point $x\in E$ and possibly further singularities. Then,  fixing coordinates $(\underline\alpha,\underline\beta)$ on $T^1_{C^\prime,x}$
as in \eqref{versal-family-m-tacnode}, we have that the image ${H^\prime_x}$ of the morphism 
$$
d\phi_{[C^\prime],x}:H^0(C^\prime,\N_{C^\prime|\mathcal X})\to T^1_{C^\prime,x}
$$
is contained in the linear space $H_x:\beta_1=\cdots=\beta_{m-1}=0.$ Moreover, it contains in codimension $1$ the image $\Gamma^\prime_x$
of $H^0(C^\prime,\N_{C^\prime|\mathcal X_0})$, which is in turn contained in $\Gamma_x$, with $\Gamma_x=T_{\underline 0}V_{1^m}:\beta_0=\beta_1=\cdots=\beta_{m-1}=0.$  
If $\Gamma^\prime_x=\Gamma_x$ then
$H^\prime_x=H_x$ and, for every $(m-1)$-tuple of integers
$d_2,\ldots,d_m$ such that $\sum_j(j-1)d_j=m-1,$ the intersection
$V_{\small 1^{d_2},\,2^{d_3},\ldots,\,{(m-1)}^{d_m}}\cap H_x^\prime$
is the curve given by \eqref{intersection}. In particular, 
there exist curves $C_t^\prime\in |D_t|$ with $d_j$ singularities of type $A_{j-1}$ at a neighborhood
of $x$, for every $j=2,...,m.$  
By \eqref{ker2}, a sufficient condition in order that $\Gamma^\prime_x=\Gamma_x$ (equivalently $H^\prime_x=H_x$)
is that $\dim(W_{x,m-1})=\dim(|D_0|)-m+1$. 
\end{corollary}
\begin{proof}
By hypothesis we have that $C^\prime=C_A^\prime\cup C_B^\prime$, with $C^\prime_A$ and $C^\prime_B$ smooth at $x$ and tangent to $E$ with multiplicity $m$.
Trivially, every deformation of $C^\prime$ in $\mathcal X_0$ has on $E$ tacnodes of order $m_i$ at points $p_i$ close to $x$ with $\sum_i m_i=m.$
In particular, every deformation of $C^\prime$ in $\mathcal X_0$ is equigeneric at $p$ \cite[Definition 3.13]{diaz_harris}. 
This implies that the image $\Gamma^\prime_x=d\phi_{[C^\prime],x}(H^0(C^\prime,\N_{C^\prime|\mathcal X_0})$ is contained in the tangent space 
$\Gamma_x=T_{\underline 0}V_{1^m}:\beta_0=\beta_1=...=\beta_{m-1}$ at $\underline 0$ of the locus $V_{1^m}$ of $m$-nodal curves.
By \eqref{kerneldifferenziale2} and the fact that
 $H^0(C^\prime,\N_{C^\prime|\mathcal X_0})$ is a codimension $1$ subspace $H^0(C^\prime,\N_{C^\prime|\mathcal X})$, 
we have that $\Gamma^\prime_x$ has codimension $1$ in $H_x^\prime.$ Moreover, if $\sigma\in H^0(C^\prime,\mathcal N_{C^\prime|\mathcal X})\setminus
H^0(C^\prime,\mathcal N_{C^\prime|\mathcal X_0})$ is the infinitesimal deformation given by \eqref{sigma}, where now 
$x+y+z^m=xy=0$ is the local equation of $C^\prime$ at $x$, then we have that the image point of $\sigma$ is contained in $H^\prime_x.$
Thus $H^\prime_x\subseteq H_x$, with $H_x:\beta_1=\cdots=\beta_{m-1}=0.$  The corollary follows now by
Proposition \ref{referee} and by versality, arguing as in the proof of Theorem \ref{prop:defcusps}.
\end{proof}


\begin{remark}\label{global-use}
Corollary \ref{ingenerale} is helpful if one wants to deform a reduced curve $C^\prime\in |D_0|$ with tacnodes at points of $E$
and further singularities, by smoothing together all singularities. In the next section, it will be applied to the case where
$\mathcal X_0$ is a reducible stable $K3$ surface and $C^\prime\subset\mathcal X_0$ is a reduced reducible curve with
nodes and tacnodes (at singular points of $\mathcal X_0$) as singularities. In \cite{galati-trip} the curve $C^\prime\subset\mathcal X_0$
is also allowed to have a space triple point at a point of $E$. In these applications the fact that the family 
$\mathcal X\to\mathbb A^1$ is a family of $K3$ surfaces is useful in the construction of the limit curve, but not in the deformation argument
of the curve, where only versality is used.

In order to help the reader understand the deformation argument in Theorem \ref{thm:maink3} (especially in the Case 1.4.2), we
want to point out the following difference of behavior between curve singularities at smooth points of $\mathcal X_0$ and at 
points of $E$. For curves on smooth surfaces we have Definition \ref{jacobian}. Assume now that the curve $C^\prime\subset\mathcal X_0$ 
has several tacnodes at points $p_1,...,p_r$ of $E$ of order $m_1,..,m_r$. Then
the natural map $$H^0(C^\prime,\N_{C^\prime|\mathcal X})\to \oplus_iH_{p_i}\subset \oplus_iT^1_{C^\prime,i}$$ is not surjective. Indeed, 
$\dim(\oplus_iH_i)=\sum_im_i$, but the kernel of the map has dimension $\geq h^0(C^\prime,\N_{C^\prime|\mathcal X})-\sum_im_i-r+1.$
This is not an obstruction to simultaneously deform  the tacnodes to desired singularities of type $A_k.$
 A sufficient condition in order to be able to smooth the tacnodes independently is that the linear system $W$ of curves in $|D_0|$ with an 
$(m_i-1)$-tacnode at $p_i$ for every $i$ has codimension $\sum_i(m_i-1)$ in $|D_0|$, as expected. Indeed, if this happens, then, by \eqref{ker2},
we have that $h^0(C^\prime,\N^\prime_{C^\prime|\mathcal X_0})=h^0(C^\prime,\N_{C^\prime|\mathcal X_0})-\sum_i (m_i-1)$.
Thus, by Corollary \ref{ingenerale}, the map  
$$H^0(C^\prime,\N_{C^\prime|\mathcal X_0})\to \oplus_i\Gamma_{p_i}$$
is surjective and hence the map $H^0(C^\prime,\N_{C^\prime|\mathcal X})\to H_{p_i}$ is a surjection for any $i$. This implies that tacnodes
may be deformed independently by arguing as in \cite[Lemma 4.4]{ch1} (see Case 1.4.2 of the proof of Theorem \ref{thm:maink3} for details).

We finally observe that the deformation argument of Theorem \ref{thm:maink3} also works if the limit curve has singularities different than nodes
on the smooth locus of $\mathcal X_0$, as long as 
$H^0(C^\prime, N^\prime_{C^\prime|\mathcal X})$ has the "expected dimension".
\end{remark}

\begin{remark}\label{parita'}
The curve \eqref{intersection} in Proposition \ref{referee} has nicer geometric properties in several cases.
Consider the case $d_j=0$ for every $j\geq 3,$ i.e. deformations of the $m$-tacnode to curves with $m-1$ nodes. Then,
by  \cite[Section 2.4]{ch}, there is only one possible conjugacy class $[\tau^+,\tau^-,\sigma]$ and 
the curve \eqref{intersection} is smooth and irreducible. Similarly, in the case $d_m=1$ and $d_j=0$ for every $j\leq m-1$,
the curve in $H_p\subset T^1_{C,p}$ parametrizing deformations of the $m$-tacnode  to an $A_{m-1}$-singularity
is defined by \eqref{A_{m-1}}. In particular it is smooth and irreducible.
This is no longer true in general. 

In the special case  that $m=2l+1$ is odd and $d_j^+=d_j^-$ for every $j,$ the parametric
equations \eqref{parametric} of \eqref{intersection} have a simpler feature, since the polynomial \eqref{gamma=1} is odd. 
Indeed, under the hypotheses, 
the polynomial  $- \nu_{(\underline d, (\tau^+,\tau^-,\sigma), 1)}(-z)$
is monic and with no $(m-1)$-term and defines an $m$-covering of $\mathbb P^1$ with monodromy in the conjugacy
class $(\tau^+,\tau^-,\sigma)$. Moreover, the discriminant polynomial $(-\nu_{(\underline d, (\tau^+,\tau^-,\sigma), 1)}(-z))^2-4$
has $d_j$ roots of multiplicity $j$, distributed as $d_j^+$ roots of $\nu_{(\underline d, (\tau^+,\tau^-,\sigma), 1)}(-z)+1$
and $d_j^-$ roots of $\nu_{(\underline d, (\tau^+,\tau^-,\sigma), 1)}(-z)-1,$ for every $j,$ and no further multiple roots.
By uniqueness, there exists $\zeta\in\mathbb C$ with $\zeta^m=1$ such that 
 $$-\nu_{(\underline d, (\tau^+,\tau^-,\sigma), 1)}(-z)=\nu_{(\underline d, (\tau^+,\tau^-,\sigma), 1)}(\zeta z),$$
 from which we deduce that $c_{i,(\underline d, (\tau^+,\tau^-,\sigma), 1)}=0,$ for every even $i$ and $\zeta=1.$
Then the equations \eqref{parametric} become 
 \begin{equation}\label{dispari}
 \begin{matrix}
\alpha_0 &=&0,&&\alpha_1=t^lc_{1,(\underline d, (\tau^+,\tau^-,\sigma), 1)},\\
\alpha_2 &=&0,&&\alpha_3=t^{l-1}c_{3,(\underline d, (\tau^+,\tau^-,\sigma), 1)}, \\
\vdots&&\vdots&&\vdots\\
\alpha_{m-2}&=& 0&\mbox{and}&\beta_0= {t^m\over 4},\,t=u^2\in\mathbb C.
\end{matrix}
\end{equation}

\end{remark}

\begin{remark}[Multiplicities and base changes]\label{multiplicity}
In the same way as in \cite[Section 1]{galati} for families of curves with only nodes and ordinary  
cusps, it is possible to define a relative Severi-Enriques variety  
$\mathcal V_{D,\small1^{d_2},\,2^{d_3},\ldots,\,{(m-1)}^{d_m}}^{\mathcal X|\mathbb A^1}\subset \mathcal H$ 
in the relative Hilbert scheme $\mathcal H,$
whose general fibre is the variety 
$\mathcal V_{D_t,\small1^{d_2},\,2^{d_3},\ldots,\,{(m-1)}^{d_m}}^{\mathcal X_t}\subset |D_t|=\mathcal H_t.$ 
Theorem \ref{prop:defcusps} proves that, whenever its hypotheses 
are verified, the locus of locally trivial deformations $W_{E,m}\subset |D_0|$ of $C$ in $|D_0|$
is one of the irreducible components of the special fibre $\mathcal V_0$ of 
$\mathcal V_{D,\small1^{d_2},\,2^{d_3},\ldots,\,{(m-1)}^{d_m}}^{\mathcal X|\mathbb A^1}\to\mathbb A^1.$ 

The multiplicity $m_C$ of $W_{E,m},$ as irreducible component of $\mathcal V_0,$ 
coincides with the intersection multiplicity at $\underline 0$ of the curve 
$$
V_{\small 1^{d_2},\,2^{d_3},\ldots,\,{(m-1)}^{d_m}}\cap H_p=B_1\cup\cdots\cup B_k
$$
with the linear space $d\phi_p(H^0(C,\N_{C|\mathcal X_0})):\beta_0=\beta_1=\cdots=\beta_{m-1}=0.$ 

Furthermore, the minimum of all intersection multiplicities at $\underline 0$ of the irreducible components $B_i$ with   
$d\phi_p(H^0(C,\N_{C|\mathcal X_0}))$ is the geometric multiplicity $m_g$ of $C$ defined in 
\cite[Problem 1 and Definition 1]{galati}. In particular, $m_g\leq m_C$. When $d_2=m-1$ and $d_j=0$ for
$j\neq 2,$ we know by \cite{ch} that $k=1$ and $B_1$ is smooth at $\underline 0$ and tangent to 
$d\phi_p(H^0(C,\N_{C|\mathcal X_0})$ with multiplicity $m_C=m.$

By Corollary \ref{ingenerale} and Remark \ref{global-use}, a similar result holds for reduced curves $C^\prime\in |D_0|$
with tacnodes on $E$ and singularities of type $A_k$ on the smooth locus of $\mathcal X_0,$ as long as 
$H^0(C^\prime, N^\prime_{C^\prime|\mathcal X})$ has the "expected dimension". In this case the multiplicity $m_{C^\prime}$
may be computed as in \cite[Lemma 4.4]{ch1}.

We finally observe that it is an abuse of terminology to say that the 
$m$-tacnodal curve $C\subset\mathcal X_0$ deforms to a curve 
$C_t\in|\mathcal O_{X_t}(D_t)|$ with $d_j$ singularities of type $A_{j-1}$ for every $j$. 
This is true only up to a base change. More precisely, let 
\begin{displaymath}
\xymatrix{
 \tilde{\mathcal X} \ar[d]\ar[r] & \mathcal X\ar[d]\\
\mathbb A^1\ar[r]^{\nu_{m_g}}&\mathbb A^1 }
\end{displaymath}
be the family of surfaces obtained from $\mathcal X\to\mathbb A^1$ by a 
base change of order $m_g$.  Observe that $\tilde{\mathcal X}$ is an $m_g$-cover
of $\mathcal X$ totally ramified along the central fibre. In particular, by substituting $\nu_{m_g}(u)=u^{m_g}=t$ in the local equation $xy=t$ of $\mathcal X$
at a point $p\in E=A\cap B,$ one finds that $\tilde{\mathcal X}$ is singular exactly along the 
singular locus $\tilde E\simeq E$ of the central fibre $\tilde{\mathcal X}_0\simeq\mathcal X_0.$  By blowing-up $m_g-1$ times $\tilde X$
along $E$ one obtains a family of surfaces
\begin{displaymath}
\xymatrix{
 \mathcal X^\prime \ar[d]\ar[r]^h & \mathcal X\ar[d]\\
\mathbb A^1\ar[r]^{\nu_{m_g}}&\mathbb A^1 }
\end{displaymath}
with smooth total space, having general fibre $\mathcal X^\prime_u\simeq \tilde{\mathcal X}_u\simeq \mathcal X_{u^m}$,
and whose central fibre has a decomposition into irreducible components 
$\mathcal X^\prime_0=\mathcal E_0\cup \mathcal E_1\cup\cdots\cup \mathcal E_{m_g-1}\cup \mathcal E_m$
where $\mathcal E_0=A$, $\mathcal E_m=B$ and $\mathcal E_i$ is a $\mathbb P^1$-bundle on the curve $\mathcal E_i\cap \mathcal E_{i-1}\simeq E$, for every $1\leq i\leq m-1$.
In this new family of surfaces $ \mathcal X^\prime\to\mathbb A^1$, we have that the pull-back curve $h^*(C)\in |\mathcal O_{\mathcal X_0^\prime}(h^*(D))|$
of $C$ deforms to a curve $C_u\in |\mathcal O_{\mathcal X_u}(h^*(D))|$ 
with the wished singularities. Equivalently, the divisor $m_gC\in |\mathcal O_{\mathcal X_0}(m_gD_0)|$ deforms to a reduced
curve in $ |\mathcal O_{\mathcal X_t}(m_gD_t)|$ having $m_g$ irreducible components, each
of which is a curve with the desired singularities.


\end{remark}

\begin{example}\label{m=5}
In the previous remark, take $m=5,$ $d_3=2$ and $d_j=0,$ for $j\neq 3.$ Then $m_C=m_g=5.$
Roughly speaking, if the curve $C$ in Theorem \ref{prop:defcusps} has a $5$-tacnode at the point $p\in E\subset\mathcal X_0$,
 then it appears as limit of curves with two ordinary cusps on $\mathcal X_t$ with multiplicity $5.$
\end{example}
\begin{proof}
If $m=5,$ $d_3=2$ and $d_j=0$ for $j\neq 3,$ then  $d_3^+=d_3^-=1$
and $(\tau^+,\tau^-,\sigma)=((123)(345)(654321))$, up to conjugation.
In particular, $V_{\small \,2^{2}}\cap H_p$ is an irreducible curve whose parametric equations are given by 
the equation \eqref{dispari} in Remark \ref{parita'}. In particular it intersects $\beta_0=\beta_1=\cdots=\beta_5=0$ with multiplicity $5$
at $\underline 0$.
We finally want to prove that  $V_{\small \,2^{2}}\cap H_p$ is smooth at $\underline 0$, by explicitly computing
\eqref{dispari}. The following argument has been suggested to us by
the referee. Let $\nu(z)$ be a degree $5$ polynomial with no $z^4$-term. If we require
$$\nu(z)=1+(z-a)^3(z-a^\prime)(z-a^{\prime\prime})=1+(z-b)^3(z-b^\prime)(z-b^{\prime\prime}),$$
by solving the corresponding polynomial equations, we obtain that
\begin{itemize}
\item $a$ is any root of $8x^5-3$, 
\item $a^\prime $ and $a^{\prime\prime} $ are the two roots of $3x^2+9a x + 8a^2$,
\item $b=-a$, $b^\prime=-a^\prime$ and $b^{\prime\prime}=-a^{\prime\prime}$. 
\end{itemize}
Thus we find that $\nu(z)=1+\frac{1}{3}(z-a)^3(3z^2+9a z + 8a^2)=z^5-\frac{10a^2}{3}z^3+5a^4z$.
In particular, the polynomial $\nu(z)$ is odd, as expected, and the equations \eqref{dispari} become
\begin{equation}\label{duecuspidi}
\alpha_0 =\alpha_2=0,\,\alpha_1=5a^4t^2,\,a_3=-\frac{10a^2}{3}t\,
\mbox{and}\,\beta_0= {t^5\over 4},\,t\in\mathbb C,
\end{equation}
with $a$ any fixed $5th$ root of $\frac{3}{8}$, proving smoothness. 
\end{proof}

The corollary below follows directly from equations \eqref{A_{m-1}} and it is an easy generalization of \cite[Lemmas 2 and 6]{galati}.

\begin{corollary}\label{d_{m}=1}
Independently of $m,$ if in the Remark \ref{multiplicity} we have that $d_{m}=1$ and $d_j=0$ for $j\neq m,$ then $m_C=m_g=2.$ 
Roughly speaking, the $m$-tacnodal curve $C$ in Theorem \ref{prop:defcusps} appears as limit of an $A_{m-1}$-singularity on
$\mathcal X_t$ with multiplicity $2.$ 
\end{corollary}

As already observed in Remark \ref{multiplicity}, Examples \ref{m=5} and Corollary \ref{d_{m}=1} may be generalized to curves 
$C^\prime\subset \mathcal X_0$ with several tacnodes on $E$ and singularities $A_k$ on the smooth locus of $\mathcal X_0$, 
under the hypothesis that $H^0(C^\prime,\N^\prime_{C^\prime|\mathcal X_0})$
has the expected dimension. But computing the multiplicity $m_{C^\prime}$ in the general case is a non-easy exercise.

\section{An application to general $K3$ surfaces }\label{S:EXIST}
This section is devoted to the proof of Theorem \ref{thm:maink3}. We also point out several corollaries of it.
 The degeneration argument we will use has been introduced in \cite{clm} and also used in \cite{C}. 
In the following $(S,H)$ will denote a general primitively polarized $K3$ surface of genus ${\p}={\p}_a(H).$
We will show  the existence of curves on $S$ with $A_k$-singularities
as an application of Theorem \ref{prop:defcusps}, more precisely of Corollary \ref{ingenerale}. In particular we will study deformations
of suitable curves with tacnodes and nodes on the union of two ad hoc constructed rational normal scrolls, cf. Remark \ref{global-use}.
We point out that our argument is strongly inspired by the one in \cite{C} but is not the same.

Let ${\p}=2l+\varepsilon\geq 3$ be an integer with $\varepsilon=0,1$ and $l \geq 1$ and
let $E \subset \PP^{{\p}}$ be an elliptic normal curve of degree ${\p}+1$. 
Consider two line bundles $L_1, L_2 \in \Pic^2(E) $ with $L_1\neq L_2$.
We denote by $R_1$ and $R_2$ the rational normal scrolls of degree ${\p}-1$ in
$\PP^{\p}$ generated by the secants of the divisors in $|L_1|$ and $|L_2|$, respectively. We have that
\[ R_i \cong  
\begin{cases} 
\PP^1 \times \PP^1 &  \; \mbox{if ${\p}=2l+1$ is odd and $\O_E(1) \not \sim (l+1)L_i$},  \\ 
 \mathbb{F}_1 &  \; \mbox{if ${\p}=2l$ is even},\\ 
\mathbb{F}_2 &  \; \mbox{if ${\p}=2l+1$ is odd and $\O_E(1) \sim (l+1)L_i$}.
\end{cases} 
\] 
We will only need to consider the first two cases. In the first case, where $R_1 \cong R_2\cong\PP^1 \times \PP^1,$ 
we let $\sigma_i$ and $F_i$ be the classes of the two rulings of $R_i$, for $i=1,2$.
In the second case, where $R_1 \cong R_2 \cong \mathbb{F}_1$, we let $\sigma_i$ be the section of negative self-intersection 
and $F_i$ be the class of a fiber. Then the embedding of $R_i$ into $\PP^{{\p}}$ is given by the line bundle 
$\sigma_i+lF_i$ for $i=1,2$ and
$R_1$ and  $R_2$ intersect transversally along the curve $E \sim -K_{R_i} \sim 2\sigma_i+(3-\varepsilon)F_i,$ which  
 is anticanonical in each $R_i$ (cf. \cite[Lemma 1]{clm}). In particular, $R:=R_1\cup R_2$ is a variety with normal 
 crossings and, by \cite[Section 2]{F}, we have that the first cotangent bundle
$T^1_R\simeq \mathcal N_{E|R_1}\otimes\mathcal N_{E|R_2}$ of $R$ is a line bundle on $E$ of degree $16.$
Let now $\U_{{\p}}$ 
 be the component of the Hilbert scheme of $\PP^{{\p}}$ 
containing $R$. Then we have that $\dim(\U_{{\p}})={\p}^2+2{\p}+19$ and, by  \cite[Theorems 1 and 2]{clm}, 
 the general point $[S]\in \U_{{\p}}$ represents a smooth, projective $K3$ surface $S$ of degree 
 $2{\p}-2$ in $\PP^{{\p}}$ such that $\Pic S \cong \ZZ[\O_S(1)]=\ZZ[H].$ 

In the proof of Theorem \ref{thm:maink3} we will consider general deformations $\mathcal{S} \to T$ of $R=\mathcal S_0$ over a
one-dimensional disc $T$ contained in $\U_{{\p}}$. 
Now $\mathcal{S}$ is smooth except for $16$ rational double points $\xi_1, \ldots, \xi_{16}$ lying on
$E$; these are the zeroes  of the section of the first cotangent bundle
$T^1_R$ of $R$ that is the image by the natural map $H^0(R,\mathcal N_{R|\mathbb P^{\p}})\to H^0(R,T^1_R)$ of the first-order 
embedded deformation determined by $\mathcal S\to T$, cf. 
\cite[pp. 644-647]{clm}. Blowing-up $\mathcal S$ at these points and contracting the obtained exceptional 
components (all isomorphic to $\mathbb F_0$)
on $R_2,$ we get a smooth family of surfaces $\mathcal X\to T,$ such that $\mathcal X_t\simeq\mathcal S_t$ and 
$\mathcal X_0=R_1\cup \tilde{R_2},$ where $\tilde{R_2}$ is the blowing-up of $R_2$ at the points  $\xi_1, \ldots, \xi_{16},$
with new exceptional curves $E_1, \ldots,\,E_{16}.$ 


\begin{proof}[Proof of Theorem \ref{thm:maink3}.] Let ${\p}=2l+\epsilon\geq 3,$ with $\epsilon=0,1,$ let $E$ be a smooth elliptic 
curve and $n\geq 1$ an integer.

{\it Case 1.} We first prove the theorem under the assumption $(n,{\p})\neq (2,3), (2,4).$ The proof will be divided into $4$ steps.

\noindent {\it Step 1.1. We construct two suitable rational normal scrolls $R_1$ and $R_2.$}
Let $L_1 \neq L_2$ be two degree $2$ line bundles on $E$ such that
\begin{equation} \label{eq:propl}
(n-1)(l-1+\varepsilon)L_1 \sim (n-1)(l-1+\varepsilon)L_2. 
\end{equation}
Note that there is no requirement if $n=1$, and that the hypothesis ${\p}\geq 5$ if $n=2$  ensures that we can choose $L_1 \neq L_2$. 
Now fix any general point $r \in E$ and embed $E$ as an elliptic normal curve of degree ${\p}+1$ by the very ample line bundle
\begin{equation} \label{eq:sezip}
\O_E(1) := \Big(2n(l-1+\varepsilon)+3-\varepsilon\Big)r-(n-1)(l-1+\varepsilon)L_i.
\end{equation}
When ${\p}=2l+1$ is odd, the condition $\O_E(1) \not \sim (l+1)L_i$ is equivalent to $(nl+1)L_i \not \sim 2(nl+1)r$, which is certainly verified for a general point $r$. 
Hence, letting $R_1$ and $R_2$ be the two rational normal scrolls in $\mathbb P^{\p}$ spanned by $L_1$ and $L_2$ as above, we have that
 $R_1 \cong R_2 \cong \PP^1 \times \PP^1$ when $\p=2l+1$ is odd and  $R_1 \cong R_2 \cong \mathbb F_1$ when $\p=2l$ is even. 

\noindent {\it Step 1.2. We next construct a special curve $C\subset R=R_1\cup R_2$, inspired by \cite{C}.}
Using the notation above, let $M_i$ be the divisor on $R_i$ defined by $M_i=\sigma_i+[n(l-1+\varepsilon)+1-\varepsilon]F_i.$  
By \eqref{eq:sezip} and the fact that ${L_i}|_{E} \sim {F_i}|_{E}$, we have that $[2n(l-1+\varepsilon)+3-\varepsilon]r \in |\O_{E}(M_i)|.$
Since
$H^0(R_i,\O_{R_i}(M_i)) 
\cong H^0(E,\O_{E}(M_i))$
we deduce that
there exists a unique (necessarily smooth and irreducible) curve $C_n^{i}\subset R_i$ such that
\begin{equation}\label{34}
  C_{n}^{i} \in \Big|M_i\Big|
\,\,\,\,
\textrm{and}\,\,\,\,
 C_{n}^{i}\cap E = \Big(2n(l-1+\varepsilon)+3-\varepsilon\Big)r,
\end{equation}
for both $i=1,2.$ Now  we fix a general point $q_0 \in E$ and we denote by
$H$ the hyperplane class of $\mathbb P^{\p}$. If $n=1$ then $M_i\sim C_{1}^{i}\sim H$ on $R_i$ and $q_0\notin  C_{1}^{i}$
for $i=1,2.$ More generally,  by \eqref{eq:propl}, \eqref{eq:sezip} and \eqref{34}, we have that 
\begin{eqnarray*}
(nH-C^1_n)|_E&\sim& 
(nH-C^2_n)|_E,
\end{eqnarray*}
for every $n\geq 1$. In particular, if $n\geq 2$,  there exists a curve $D^1\cup D^2\subset R$ with
$D^i=\cup_{j=1}^{n-1}C^i_j\subset R_i,$ where every $C_j^{i} \in |\sigma_i+(1-\varepsilon)F_i|$  
is a (necessarily  smooth and irreducible) curve on $R_i$, such that 
\[
C_{j}^{1}\cap E = q_{2j-2}+(2-\varepsilon)q_{2j-1} \; \; \mbox{and} \; \; 
C_{j}^{2}\cap E = (2-\varepsilon)q_{2j-1}+q_{2j},
\]
with $i=1,2$ and $1 \leq j \leq n-1,$ and where $q_1,q_2,\ldots,q_{2n-2}=q_0$ are distinct points on $E.$ 
Notice that the curve $D^1\cup D^2\subset R$ is uniquely determined by $q_0$ if ${\p}$ is odd, while for ${\p}$
even there are finitely many curves like $D^1\cup D^2$.
Now, for $i=1,2$, let $C^i\in |\mathcal O_{R_i}(nH)|$ be the curve defined by
 \begin{equation*}
 C^i=\left\{\begin{array}{l}
 C^i_1\,\,\mbox{if}\,\, n=1,\\
D^i\cup C^i_n= C_{1}^{i}\cup C_{2}^{i}\cup \cdots \cup
C_{n-1}^{i}\cup C_{n}^{i}\,\,\mbox{if}\,\, n\geq 2.
\end{array}\right.
\end{equation*}
 Observe that, if $n\geq 2$, because of the generality of $q_0,$ we may assume that all irreducible components $C_{j}^{i}$ of $D^i$  intersect
$C^i_n$ transversally for $i=1,2.$ In particular, we have that the singularities of $C:=C^1\cup C^2\in |\mathcal O_{R}(nH)|$
consist of a $({\p}+1)$-tacnode at $r\in E$ if $n=1$, and are given by nodes on $R\setminus E$ and nodes and tacnodes on $E$ if $n\geq 2$.

\noindent {\it Step 1.3. We now construct a general deformation  $\mathcal{S} \to T$ of $R,$ whose general fibre $\mathcal S_t$ is a 
smooth projective K3 surface, and a smooth birational modification of it
\begin{equation}
\xymatrix{
\mathcal X  \ar[dr]\ar[r]^\pi & \mathcal S\ar[d]\\
& T}
\end{equation} as above}. Let $\xi_1+\cdots+\xi_{16}\in |T^1_R|$ be a general divisor 
if $n=1$ and a general divisor such that 
$\xi_1=q_0$ and $\xi_l\neq q_m,$ for every $l\geq 2$ and $m\geq 1,$ if $n\geq 2.$ By the generality of $q_0,$
we have that $\xi_1+\cdots+\xi_{16}$ is a general member of $|T^1_R|$ also for $n\geq 2.$ By the surjectivity of
the natural map $H^0(R,\mathcal N_{R|\mathbb P^{{\p}}})\to H^0(R,T^1_R)$ (see \cite[Cor.~1]{clm}),
by \cite[Theorems 1 and 2]{clm} and related references (precisely, \cite[Remark 2.6]{F}
and \cite[Section 2]{GH}), we deduce that there exists a deformation $\mathcal{S} \to T$ of $\mathcal S_0=R$ whose general fiber
is a smooth projective  $K3$ surface $\mathcal S_t$ in $\mathbb P^{{\p}}$ with 
$\Pic (\mathcal S_t) \cong \mathbb Z [\mathcal O _{\mathcal S_t}(1)]\cong \mathbb Z [H]$ and 
such that $\mathcal S$ is singular exactly at the
 points $\xi_1\ldots,\xi_{16}\in E.$ Let $\mathcal X\to T$ be the smooth family obtained from $\mathcal{S} \to T$ 
 as above and $\pi:\mathcal X\to\mathcal S$ the induced birational morphism.  We recall that it has special fibre 
 $\mathcal X_0=R^1\cup \tilde{R^2}$, where $ \tilde{R^2}$ is the blowing up of $R_2$ at $\xi_1\ldots,\xi_{16}$
 and  $R_1\cap \tilde{R^2}=E$.
 
\noindent {\it Step 1.4. Let $\tilde C$ and $\pi^*(C)$ be the strict transform  and the pull-back of $C$ with respect to  
the natural morphism $\pi:\mathcal X\to\mathcal S.$ Using the ideas developed in Section 3, we will prove that $\pi^*(C)$ 
deforms into a family of curves in the $\mathcal X_t$'s
enjoying the required properties.} 

In the case $n=1$, the result is a straightforward application of Theorem \ref{prop:defcusps}.
Indeed, for $n=1$, by the generality of $r\in E$, the curve $\tilde C\simeq\pi^*(C)$ is a $({\p}+1)$-tacnodal curve
satisfying all  hypotheses of Theorem \ref{prop:defcusps}. 

Assume now that $n\geq 2$. In this case $\pi^*(C)=\tilde C\cup E_{q_0}$, where $E_{q_0}=E_1\subset\tilde{R_2}$ is
the $(-1)$-curve corresponding to $q_0=\xi_1.$  By abusing notation, we denote every irreducible component of $\tilde C$
as the corresponding irreducible component of $C.$ In particular, we set $\tilde C= C_n^1\cup D^1\cup C_n^2\cup D^2$
and $D^i=\cup_{j=1}^{n-1} C^i_j,$ for every $i$. The singularities of $\pi^*(C)$ on $\mathcal X_0\setminus E$ are given by the 
singularities of the strict transform $\tilde C$ of $C$ and a further node at $x_0=E_{q_0}\cap \tilde C.$ We want to 
obtain curves in $|\mathcal O_{\mathcal X_t}(nH_t)|$ with the desired singularities as deformations of 
$\pi^*(C).$ We first observe that
every deformation of $\pi^*(C)$ in  $|\mathcal O_{\mathcal X_t}(nH_t)|$ is an irreducible curve. 
This may easily be verified using that the divisor $H_t$ generates $\Pic(\mathcal X_t).$ 
In particular, no matter how we deform $\pi^*(C)$ to a curve on $\mathcal X_t,$ at least one node  of $\pi^*(C)$
on $(C^1_n\cup C^2_n)\cap (D^1\cup D^2)\cap (R\setminus E)$ must be smoothed. Moreover 
the smoothed node may be chosen arbitrarily, as will be clear by the following argument.
Let $q\in (C^1_n\cup C^2_n)\cap (D^1\cup D^2)\cap (\mathcal X_0\setminus E)$ be any fixed point and consider the 
natural morphism
$$
\Phi:H^0(\pi^*(C),\N_{\pi^*(C)|\mathcal X})\to T=
\oplus_{x\,\,\textrm{tacnode}}T^1_{\pi^*(C),x}\oplus_{y\neq q\,\,\textrm{node on}\,\, \mathcal X_0\setminus E}T^1_{\pi^*(C),y}
$$
obtained by composing the morphism $H^0(\pi^*(C),\N_{\pi^*(C)|\mathcal X})\to H^0(\pi^*(C),T^1_{\pi^*(C)})$ with
the projection $H^0(\pi^*(C),T^1_{\pi^*(C)})\to T.$ By Remark \ref{A_k} and Lemma \ref {lemma-ker}, the kernel of $\Phi$
is contained in the subspace $r_{\pi^*(C)}(\mathcal W)\subset H^0(\pi^*(C),\N_{\pi^*(C)|\mathcal X_0})$,
where 
$r_{\pi^*(C)}:H^0(\mathcal X_0,\mathcal O_{\mathcal X_0}(\pi^*(C)))\to H^0(\pi^*(C),\mathcal N_{\pi^*(C)|\mathcal X_0})$ is the restriction map 
and $\mathcal W\subset H^0(\mathcal X_0,\mathcal O_{\mathcal X_0}(\pi^*(C)))$ is the subspace associated to the
 linear system $W\subset |\mathcal O_{\mathcal X_0}(\pi^*(C))|$ of curves passing through every node $y\neq q$
of $\pi^*(C)$ on $\mathcal X_0\setminus E$ and having an $(m-1)$-tacnode
at every $m$-tacnode of $\pi^*(C)$. 

We now want to show that
\begin{equation}\label{nuova}
\ker(\Phi)=r_{\pi^*(C)}(\mathcal W)=H^0(\pi^*(C),\mathcal N^\prime_{\pi^*(C)|\mathcal X_0})=H^0(\pi^*(C),\mathcal N^\prime_{\pi^*(C)|\mathcal X})=\{0\}.
\end{equation}
The equality $H^0(\pi^*(C),\mathcal N^\prime_{\pi^*(C)|\mathcal X_0})=H^0(\pi^*(C),\mathcal N^\prime_{\pi^*(C)|\mathcal X})$
follows from Lemma \ref{lemma-ker}. Moreover, we observe that, in order to prove \eqref{nuova}, it is enough to show that
the linear system $W$ consists of the unique curve $W=\{\pi^*(C)\}$. Indeed, if this last equality is
true, then $\ker(\Phi)=r_{\pi^*(C)}(\mathcal W)=\{0\}.$ Moreover the equality $W=\{\pi^*(C)\}$ also implies that
every curve $D$ in $W$ contains the point $q$, too. In other words, every infinitesimal deformation of $\pi^*(C)$ in $\mathcal X_0$ preserving every tacnode 
and the nodes $y$ different from $q$ on $\mathcal X_0\setminus E$, also preserves $q$. Using that the nodes of $\pi^*(C)$ on $E$ are trivially
preserved by every section of $H^0(\pi^*(C),\N_{\pi^*(C)|\mathcal X_0}),$ we have that 
$r_{\pi^*(C)}(\mathcal W)\subset H^0(\pi^*(C),\mathcal N^\prime_{\pi^*(C)|\mathcal X_0}).$ Since the other inclusion holds trivially,
we obtain \eqref{nuova}. 

Now the equality $W=\{\pi^*(C)\}$  is a straighforward application of Bezout's theorem. 
Let $B=B_1\cup B_2$, where $B_1=B|_{R_1}$ and $B_2=B|_{\tilde{R_2}}$, be a curve in the linear system $W$.
Assume that $q\in\tilde{R_2}$. The other case is similar. We first observe that the intersection number $B_i.C_n^i$ is given by 
$B_i.C_n^i=n^2(l-1+\varepsilon)+nl.$ Moreover, by the hypothesis $B\in W$, we have that the intersection multiplicity of $B_1$ and $C_n^1$ at $r$ is given by 
${\rm{mult}}_r(B_1,C_n^1)=2n(l-1+\varepsilon)+2-\varepsilon$. Furthermore the intersection $B_1\cap C^1_n$ contains the intersection points
$C^1_n\cap C^1_j$, for $1\leq j\leq n-1.$ We deduce that the cardinality of the intersection $B_1\cap C^1_n$ is at least equal to
$$
2n(l-1+\varepsilon)+2-\varepsilon+(n-1)M_i\big(\sigma_i+(\varepsilon-1)F_i\big)= n^2(l-1+\varepsilon)+nl+1.
$$
Thus, by Bezout's Theorem, $C_n^1\subset B_1.$ Since $B$ is a Cartier divisor, it follows that the intersection multiplicity of $B_2$ 
with $E$ at $r$ is given by ${\rm{mult}}_r(B_2,E)={\rm{mult}}_r(B_1,E)\geq {\rm{mult}}_r(C_n^1,E)=2n(l-1+\epsilon)+3-\epsilon$. Moreover $B_2$ contains
the points, different from $q$, arising from the intersection of $C_n^2\cap C_j^2$, with $j\leq n-1$. Using again Bezout's theorem,
we find that $C_n^2\subset B_2$ and, in particular $q\in B.$ It remains to prove that $C^i_j\subset B^i$, for $i=1,2$ and $j\leq n-1.$ 
We observe that, if $j\leq n-1$, then $B_i.C^i_j=nl.$ Now consider the intersection $B_2\cap C_{n-1}^2$. It contains the point $q_{2n-3}$
with multiplicity $1-\epsilon$; the point $x_0$; the $(n-2)(1-\epsilon)$ points arising from the intersection $C^2_j\cap C_{n-1}^2$, for $j\leq n-2$;
the intersection points $C^2_{n-2}\cap C_n^2$. This amounts to a total of $nl+1$ points. Thus, $C_{n-1}^2\subset B_2$. 
Similarly we have that $E_{q_0}\subset B_2$. Then $B_1$ passes from the further points $q_0=q_{2n-2}=E_{q_0}\cap E$ and $q_{2n-3}$ with 
multiplicity $2-\epsilon.$ This implies by Bezout that $C_{n-1}^1\subset B_1$. Applying this argument $2(n-2)$ more times,
one obtains that $B=\pi^*(C)$ and thus  $\Phi$ is injective. 

The rest of the proof will be divided according to the parity of $\p.$

\noindent {\it Case 1.4.1: ${\p}=2l+1$ is odd. } Then all singularities of $\pi^*(C)$ are nodes except for the point $r\in E$ that is a
 $(2nl+2)$-tacnode.  Moreover, by Corollary \ref{ingenerale} (and using the
 notation therein), the image of  $\Phi$ is contained in the linear space 
 $$T^\prime=H_r \oplus_{y\neq q\,\,\textrm{node on}\,\, R\setminus E}T^1_{\pi^*(C),y}\subset T.$$
 As $\Phi$ is injective, the image of $\Phi$ has dimension 
$h^0(\pi^*(C),\N_{\pi^*(C)|\mathcal X})=2n^2l+2$. Since 
$\dim(H_r)=2nl+2$ by Corollary \ref{ingenerale} and  
 the curve $\pi^*(C)$ has exactly $2nl(n-1)$ nodes on $R\setminus E$ different from $q$, the image of $\Phi$ must coincide with $T^\prime.$ 
 Again by Corollary \ref{ingenerale} and by versality, we deduce that the curve $\pi^*(C)$ may be deformed to a curve $C_t\in |\mathcal O_{\mathcal X_t}(nH_t)|$
 preserving all nodes of $\pi^*(C)$ on $R\setminus E$ except $q$ and deforming the $(2nl+2)$-tacnode
 to $d_k$ singularities of type $A_{k-1},$ for every sequence $(d_k)$ of non-negative integers such that $\sum_k(k-1)d_k=2nl+1.$
Moreover, by the fact that $\ker(\Phi)=H^0(\pi^*(C),\mathcal N^\prime_{\pi^*(C)|\mathcal X})=\{0\},$ 
we obtain that the family of curves $C_t\in |\mathcal O_{\mathcal X_t}(nH_t)|$ constructed in this way is, scheme theoretically, a generically 
smooth curve $\mathcal B\subset\mathcal H_{\mathcal X|\mathbb A^1}$ in the relative Hilbert scheme. 
 By the openness of versality (more precisely, by the properties \cite[(3.5), (3.6)]{diaz_harris} of versal deformation families), if $[C_t]\in\mathcal B$ is a general point,
 then $T_{[C_t]}\mathcal B\simeq H^0(C_t,\N^\prime_{C_t|\mathcal X})$. In particular we obtain that
$$\dim(T_{[C_t]}ES(C_t))=h^0(C_t,\N^\prime_{C_t|\mathcal X_0})=h^0(C_t,\N^\prime_{C_t|\mathcal X})-1=0.$$
 This proves the theorem in the case ${\p}$ is odd. 
 
 {\it Case 1.4.2: ${\p}=2l$ is even.} In this case $\pi^*(C)$ has a $(2n(l-1)+3)$-tacnode at $r,$ a $2$-tacnode
 at $q_{2j-1},$ for every $j=1,\ldots,n-1,$ and nodes elsewhere. In particular, $\pi^*(C)$ has $2(n-1)(nl-n+1)+(n-1)(n-2)$ nodes
 on $R\setminus E$ different from $q$. Again by Corollary \ref{ingenerale} 
(and using the notation therein),  
 $\Phi(H^0(\pi^*(C),\N_{\pi^*(C)|\mathcal X}))$ is contained in the linear subspace 
 $$ T^\prime=H_r\oplus_{j=1}^{n-1}H_{q_{2j-1}}\oplus_{y\neq q\,\,\textrm{node on}\,\, R\setminus E}T^1_{\pi^*(C),y}$$
 of $T.$ In this case the image of $\Phi$ does not coincide with $T^\prime$ unless $n=1.$ For $n=1$ the theorem follows as before.
 If $n>1$ we observe that  $\Phi(H^0(\pi^*(C),\N_{\pi^*(C)|\mathcal X}))$ contains 
$\Phi(H^0(\pi^*(C),\N_{\pi^*(C)|\mathcal X_0}))$
as a codimension $1$ subspace. Moreover, by Corollary \ref {ingenerale} and a straightforward
 dimension count, we have that 
$$\Phi(H^0(\pi^*(C),\N_{\pi^*(C)|\mathcal X_0}))=\Gamma_r\oplus_{j=1}^{n-1}\Gamma_{q_{2j-1}}\oplus_{y\neq q\,\,\textrm{node on}\,\, R\setminus E}T^1_{\pi^*(C),y},$$
 where $\Gamma_r\subset H_r$ is the locus of $(2n(l-1)+3)$-nodal curves and $\Gamma_{q_{2j-1}}\subset H_{q_{2j-1}}$
 is the locus of $2$-nodal curves, for every $j=1,\ldots,n-1,$ cf. Remark \ref{remark-im}.  It follows that 
$$\Phi(H^0(\pi^*(C),\N_{\pi^*(C)|\mathcal X}))=\Omega\oplus_{y\neq q\,\,\textrm{node on}\,\, R\setminus E}T^1_{\pi^*(C),y},$$ 
where $\Omega\subset H_r\oplus_{j=1}^{n-1}H_{q_{2j-1}}$ is a linear subspace containing $\Gamma_r\oplus_{j=1}^{n-1}\Gamma_{q_{2j-1}}$ as a codimension $1$ subspace. 
Moreover, again by Corollary \ref {ingenerale} and the surjectivity of the map $H^0(\pi^*(C),\N_{\pi^*(C)|\mathcal X_0}))\to\Gamma_{x}$, with $x=r$ or $x=q_{2j-1}$,
we have that the projection maps
$$\rho_r:\Omega\to H_r\,\,\mbox{and}\,\,\rho_{q_{2j-1}}:\Omega\to H_{q_{2j-1}}$$ are surjective, for every $j$.
By \cite[Section 2.4]{ch}, the locus of $1$-nodal curves in $H_{q_{2j-1}}$ is a smooth curve simply tangent to $\Gamma_{q_{2j-1}}$ at $\underline 0.$ 
Let $(d_2,\ldots,d_m)$ be any $(m-1)$-tuple of non-negative
 integers such that $\sum_{k=2}^m(k-1)d_k=2n(l-1)+2.$  By Proposition \ref{referee} again, the locus  
 $V_{\small 1^{d_2},\,2^{d_3},\ldots,\,{(m-1)}^{d_m}}\subset H_r$ of points corresponding
 to curves with $d_k$ singularities of type $A_{k-1},$ for every $k,$ is a reduced (possibly
 reducible) curve intersecting $\Gamma_r$ only at $\underline 0.$ It follows that the locus  of curves in $\Omega$ with $d_2+n-1$ nodes and 
 $d_k$ singularities of type $A_{k-1}$ for every $k\geq 3$   
   is a reduced (possibly reducible) curve. Parametric equations of this curve may be explicitly 
   computed for selected values of $d_2,\ldots, d_m$ (see for example Remark \ref{parita'}, Example \ref{m=5} and Corollary \ref{d_{m}=1}) 
   by arguing exactly as in \cite[proof of Lemma 4.4, p.~381-382]{ch1}.
  By versality, the curve $\pi^*(C)$ may be deformed to a curve $C_t\in |\mathcal O_{\mathcal X_t}(nH_t)|,$
 preserving all nodes of $\pi^*(C)$ on $R\setminus E$ except $q$ and  deforming 
 every simple tacnode of $\pi^*(C)$ to a node and the $(2n(l-1)+3)$-tacnode at $r$ to 
 $d_k$ singularities of type $A_{k-1}$. As before, by the fact that $h^0(\pi^*(C),\N^\prime_{\pi^*(C)|\mathcal X})=0$
 and properties \cite[(3.5), (3.6)]{diaz_harris} of versal deformation families, we obtain that 
 $\dim(T_{[C_t]}ES(C_t))=0,$ for a general $t.$
 
{\it Case 2.}  We finally consider the cases $(n,{\p})=(2,3)$ and $(2,4).$ 
 Let $E$ be a general elliptic normal curve of degree ${\p}+1$ in $\mathbb P^{{\p}}$ and $R_1=Q_1$
 and $R_2=Q_2$ be two general rational normal scrolls intersecting  transversally  along $E.$ 
 Let $X\to T$ be a one-parameter family of very general primitively polarized $K3$ surfaces with special fibre $X_0$
 and double points $p_1,\ldots,p_{16}\in E.$ Consider on $X_0:=Q_1\cup Q_2$
 the curve $\cup_{i=1}^2\cup_{j=1}^n C^i_j$ constructed in \cite[Section 3.2]{C}. The theorem
 follows in this case by studying deformations of this curve by the same techniques as before.
 Details are left to the reader.
  \end{proof}

In the case $n=1$, Theorem \ref{thm:maink3} proves that the variety $\mathcal V_{H,\small1^{d_2},\,2^{d_3},\ldots,\,{(m-1)}^{d_m}}^{S}$
is non-empty whenever it has non-negative expected dimension. The precise statement is the following.

\begin{corollary}\label{H}
Let $(S,H)$ be a general primitively polarized $K3$ surface of genus ${\p}={\p}_a(H).$ Then,
for every $(m-1)$-tuple of non-negative integers $d_2,\ldots,d_m$ such that
 \begin{equation}\label{optimal-existence}
\sum_{j=2}^m(j-1)d_j\leq\dim(|H|)=\p,
\end{equation}
there exist reduced irreducible curves $C$ in the linear system $|H|$ on $S$ 
 having $d_j$ singularities of type $A_{j-1}$ for every $j=2,\ldots,m,$ and no further singularities and
corresponding to regular points of their equisingular deformation locus $ES(C_t).$ Equivalently, 
$\dim(T_{[C_t]}ES(C_t))=\dim(|H|)-\sum_{j=2}^m(j-1)d_j.$ In particular, the variety $\mathcal V_{H,\small1^{d_2},\,2^{d_3},\ldots,\,{(m-1)}^{d_m}}^{S}$
is non-empty whenever \eqref{optimal-existence} is satisfyied.
\end{corollary}

The previous result is optimal for $g=0$, because, by \cite{C1}, we know that all the rational curves in $|H|$ are nodal
(and nodes are the worst expected singularities of a rational curve in $|H|$). 
Theorem \ref{thm:maink3}  also proves the existence of divisors in  $\mathcal V_{nH,1^\delta}^S$, parametrizing curves with a tacnode 
or a cusp and nodes, whenever they have non-negative expected dimension.
  
 \begin{corollary}\label{divisori}
Let $(S,H)$ be a general primitively polarized $K3$ surface of genus ${\p}={\p}_a(H)\geq 3$ and let
$\delta\leq\p -1$. 
Then the Severi variety $\mathcal V_{nH,1^\delta}^S$ of reduced and irreducible $\delta$-nodal curves
 contains two non-empty generically smooth divisors $V_{tac}$ and $V_{c},$
whose general point of every irreducible component corresponds to a curve with a simple tacnode 
and $\delta-2$ nodes and an ordinary cusp and $\delta-1$ nodes, respectively.
In particular the varieties $\mathcal V_{nH,{1}^{\delta-2},3^1}^S$ and $\mathcal V_{nH,{1}^{\delta-1},2^1}^S$,
are non-empty. 
\end{corollary}

\begin{remark}
The existence of a further non-empty generically smooth divisor $V_{trip}\subset \mathcal V_{nH,1^\delta}^S$,
whose general element in every irreducible component corresponds to a curve with a triple point and $\delta-3$ nodes, has been
proved in \cite[Corollary 4.2]{galati-trip} under the assumption $(n,{\p})\neq (1,4)$. The case $(n,{\p})=(1,4)$ has been studied in 
\cite[Proposition 2.2]{galati-knutsen}. It is unknown if $\mathcal V_{nH,1^\delta}^S$ may contain divisors $W$ different from $V_{trip}$, $V_{tac}$
and $V_{c}$ and parametrizing curves with singularities different than nodes.
\end{remark}

We finally observe that Theorem \ref{thm:maink3} together with Proposition \ref{prop:regolarita-gen}  provide sufficient conditions
for the variety $\mathcal V_{nH,\small1^{d_2},\,2^{d_3},\ldots,\,{(m-1)}^{d_m}}^{S}$ to be non-empty and regular.
In the following remark we write explicitly the corresponding 
existence and regularity condition for $n=1$. The case $n\geq2$ is left to the reader. 

\begin{remark}\label{ex-reg-1}
 By Theorem \ref{thm:maink3} and Proposition \ref{prop:regolarita-gen} in the case $n=1$,
we have that, if $(S,H)$ is a general primitively polarized $K3$ surface of genus ${\p}$ and
 \begin{equation}\label{optimal-regularity}
\sum_{j=2}^m(j-1)d_j< \frac{\p+2}{2},
\end{equation}
then the variety $\mathcal V_{H,\small1^{d_2},\,2^{d_3},\ldots,\,{(m-1)}^{d_m}}^{S}\subset |H|$ is non-empty and regular (cf. Definition
\ref{definition-regularity} and Remark \ref{A_k}).
This condition of existence and regularity 
 is certainly an improvement of \eqref{regularity_keilen} but is not optimal. Indeed, by Mumford and Tannebaum \cite{MM} and \cite{tan},
  we know that the Severi variety of $\delta$-nodal curves in $|H|$ is non-empty and regular for every $\delta\leq\p.$
When \eqref{optimal-regularity} is not satisfied, the existence
of irreducible components $V\subset  \mathcal V_{H,\small1^{d_2},\,2^{d_3},\ldots,\,{(m-1)}^{d_m}}^{S}$ of dimension bigger that the expected
would imply the reducibility of the variety.
\end{remark}


\let\appendix\section
\begin{appendices}
\appendix{Proof of Lemma \ref{tricicli}}\label{tricicli_proof}  
The proof is by induction on $m$.

{\it Base case of the induction}. We first prove the lemma in the special case of an admissible $2(n-1)$-tuple satisfying 
$\sum_{j=2}^{n}d_j^-=1$ or $\sum_{j=2}^{n}d_j^+=1$. So assume, by symmetry, that $\sum_{j=2}^{n}d_j^-=1$. 
Then there is an index $i_0$ such that $d^-_{i_0}=1$ and $d^-_j=0$ for all $j \neq i_0$. The question is whether there is a permutation  $\tau^+$ of cyclic structure $\Pi_{j=2}^n j^{d_j^+}$ and a cycle $\tau^-=\sigma_{i_0}$ of order $i_0$ such that 
$\tau^+ \sigma_{i_0}$ is cyclic of order $m.$ Let $\sigma_{i_0}=(1 \; 2 \; \cdots \; i_0)$. 
By \eqref{eq:amm1}, we have $\sum_{j=2}^{n}(j-1)d_j^+=m-i_0$. 
This implies that we can construct a permutation $\tau^+$ of the desired cyclic structure such that each cycle contains precisely one integer in the set $\{1,2,\ldots,i_0\}$. It is then easily seen that $\tau^+ (1 \; 2 \; \cdots \; i_0)$ is an $m$-cycle.

{\it Induction step}. The base cases of the induction are all cases where 
$\sum_{j=2}^{n}d_j^-=1$ or $\sum_{j=2}^{n}d_j^+=1$, which have been treated above. Now let $(d_2^+,d_2^-,\ldots, d_n^+,d_n^-)$ be an admissible  $2(n-1)$-tuple such that both $\sum_{j=2}^{n}d_j^{\pm} \geq 2$. By symmetry we may assume that 
$\sum_{j=2}^{n}jd_j^{-}\geq \sum_{j=2}^{n}jd_j^{+}.$ Set $i_0:=\min\{j \; | \; d_j^{-}>0\}.$ We  claim that the $2(n-1)$-tuple
\[ ({d'}_2^+,{d'}_2^-,\ldots, {d'}_{i_0}^{+}, {d'}_{i_0}^{-}, \ldots, {d'}_n^+,{d'}_n^-)= 
(d_2^+,d_2^-,\ldots, d_{i_0}^{+}, d_{i_0}^{-}-1, \ldots, d_n^+,d_n^-) \]
is admissible. Indeed, set $m':= \sum_{j=2}^{n}(j-1)({d'}_j^++{d'}_j^-)+1= m-i_0+1.$
Then $m' \geq 2$ since $ \sum_{j=2}^{n}{d'}_j^+=\sum_{j=2}^{n}d_j^+ \geq 2.$
Clearly, 
\begin{equation}
  \label{eq:strict}
\displaystyle\sum_{j=2}^{n}j{d'}_j^{-} =\displaystyle\sum_{j=2}^{n}jd_j^{-} 
-i_0 \leq m-i_0< m-i_0+1=m' 
\end{equation}
by \eqref{eq:amm2}. Assume that
$\sum_{j=2}^{n}j{d'}_j^{+} =\sum_{j=2}^{n}jd_j^{+} > m'=m-i_0+1.$ Then we have that
\[ m-i_0+2 \leq \displaystyle\sum_{j=2}^{n}jd_j^{+} \leq \displaystyle\sum_{j=2}^{n}jd_j^{-},\,\,\,\textrm{whence} \]
\begin{eqnarray*}
  2m-2i_0+4 & \leq & \displaystyle\sum_{j=2}^{n}j(d_j^{+}+d_j^{-}) = 2\displaystyle\sum_{j=2}^{n}(j-1)(d_j^{+}+d_j^{-})- \displaystyle\sum_{j=2}^{n}(j-2)(d_j^{+}+d_j^{-})
\\ & = & 2(m-1) -\displaystyle\sum_{j=2}^{n}(j-2)(d_j^{+}+d_j^{-}),
\end{eqnarray*}
by \eqref{eq:amm1}. It follows that
$0 \leq \displaystyle\sum_{j=2}^{n}(j-2)d_j^{-} \leq \displaystyle\sum_{j=2}^{n}(j-2)(d_j^{+}+d_j^{-}) \leq 2i_0-6. $
In particular, we obtain that $i_0 \geq 3.$ Moreover, by definition of $i_0$,  we must have
$ 2(i_0-2) \leq (i_0-2) \displaystyle\sum_{j=2}^{n}d_j^{-} \leq \displaystyle\sum_{j=2}^{n}(j-2)d_j^{-} \leq 2i_0-6, $
getting a contradiction. Therefore, we have proved our claim that the $2(n-1)$-tuple  
\[ ({d'}_2^+,{d'}_2^-,\ldots, {d'}_{i_0}^{+}, {d'}_{i_0}^{-}, \ldots, {d'}_n^+,{d'}_n^-)= 
(d_2^+,d_2^-,\ldots, d_{i_0}^{+}, d_{i_0}^{-}-1, \ldots, d_n^+,d_n^-) \]
is admissible. 

 By induction, there exist permutations $\tau^{\pm}$ in the symmetric group
$\mathfrak{S}_{m-i_0+1}$ of order $m-i_0+1$ of cyclic structures $\Pi_{j=2}^n j^{{d'}_j^{\pm}}$, respectively, such that
\[ \tau^+ \tau^-= \Big(1 \; 2 \; \cdots (m-i_0+1)\Big). \]
The number of distinct integers from $\{1,2,\ldots,m-i_0+1\}$ appearing in the permutation 
$\tau^-$ is $\sum_{j=2}^{n}j{d'}_j^{-}$, which is less than $m-i_0+1$ by 
\eqref{eq:strict}. Hence there exists an $x \in \{1,2,\ldots,m-i_0+1\}$ not appearing in 
$\tau^-$. Then the permutation 
\[\alpha^-= \tau^- \Big( (m-i_0+2) \; (m-i_0+3) \; \cdots \;m\; x \Big) \]
has cyclic structure $\Pi_{j=2}^n j^{d_j^{-}}$
and
\[ 
\tau^+ \alpha^- = 
\Big(1 \; 2 \; \cdots  \; (m-i_0+1)\Big) \Big( (m-i_0+2)\; (m-i_0+3) \; \cdots \; m\; x \Big) 
\]
is cyclic of order $m$, as desired. \hfill$\square$

\begin{remark}
In general, given an admissible $(2n-2)$-tuple, we have several conjugacy classes of triples of permutations  satisfying Lemma \ref{tricicli}. For example, if $m=7$, $d_2^+=d_3^+=1,$  $d_4^-=1$ and $d_j^\pm=0$ otherwise, 
then the two triples $((267)(15),(1234),(1672345)^{-1} )$ and $((365)(17),(1234),(1256347)^{-1})$ satisfy Lemma \ref{tricicli}
and are not conjugated.
\end{remark}

\end{appendices}


{}
\end{document}